\documentclass[12pt]{article}

\usepackage[margin=1in]{geometry}
\usepackage{amsthm}
\usepackage{amssymb}
\usepackage{amsmath}
\usepackage{graphicx}
\usepackage{mathdots}
\usepackage{mathtools}
\usepackage{url}
\usepackage{color}
\usepackage{framed}
\usepackage{tikz}
\usepackage{stmaryrd}
\usepackage[ruled,vlined]{algorithm2e}
\usepackage{array}

\usetikzlibrary{cd}

\definecolor{darkgreen}{rgb}{0.0, 0.5, 0.0}

\tikzcdset{
	arrow style=tikz,
	diagrams={>={Straight Barb[scale=0.8]}}
}

\newcommand{\llangle}{\langle\hspace{-2.5pt}\langle}
\newcommand{\rrangle}{\rangle\hspace{-2.5pt}\rangle}
\newcommand{\dslash}{/\hspace{-2.5pt}/}

\newtheorem{theorem}{Theorem}%[section]
\newtheorem{corollary}[theorem]{Corollary}
\newtheorem{lemma}[theorem]{Lemma}
\newtheorem{proposition}[theorem]{Proposition}

\theoremstyle{definition}
\newtheorem{problem}[theorem]{Problem}
\newtheorem{example}[theorem]{Example}
\newtheorem{definition}[theorem]{Definition}

\title{{Towards a bilipschitz invariant theory}}

\author{
Jameson~Cahill\footnote{Department of Mathematics and Statistics, University of North Carolina Wilmington, Wilmington, NC} 
\qquad
Joseph~W.~Iverson\footnote{Department of Mathematics, Iowa State University, Ames, IA} 
\qquad
Dustin~G.~Mixon\footnote{Department of Mathematics, The Ohio State University, Columbus, OH} \footnote{Translational Data Analytics Institute, The Ohio State University, Columbus, OH}
}

\date{}

\begin{document}
\maketitle

\begin{abstract}
Consider the quotient of a Hilbert space by a subgroup of its {automorphisms}.
We study whether this orbit space can be embedded into a Hilbert space by a bilipschitz map, and we identify constraints on such embeddings.
\end{abstract}

\section{Introduction}

Given a Hilbert space and a subgroup of {its automorphisms (isometric linear bijections)}, we wish to embed the {corresponding} orbit space into a Hilbert space by a bilipschitz map.
We are motivated by the analysis of data that resides in such an orbit space.
For example:

\medskip
\noindent
\textbf{Graphs.}
The adjacency matrix of an unlabeled graph on $n$ vertices is a member of $\mathbb{R}^{n\times n}$, but this matrix representation is only unique up to the conjugation action of $S_n$.

\medskip
\noindent
\textbf{Point clouds.}
A point cloud consisting of $n$ vectors in $\mathbb{R}^d$ can be represented as a member of $\mathbb{R}^{d\times n}$ up to the right action of $S_n$ (i.e., column permutation).

\medskip
\noindent
\textbf{Landmarks.}
Landmarks of a biological specimen can be represented in $\mathbb{R}^{3\times n}$, with the left action of $SO(3)$ giving different representations of the same specimen.

\medskip
\noindent
\textbf{Audio signals.}
An audio signal can be modeled as a real-valued function in $L^2(\mathbb{R})$, but different time delays of the same signal arise from an orthogonal action of $\mathbb{R}$.

\medskip

In each of the settings above, data is represented in a Hilbert space $V$, but each data point is naturally identified with the other members of its orbit {under a group $G$ of automorphisms of $V$}.
To correctly measure distance, it is more appropriate to treat the data as residing in the orbit space {$V/G:=\{G\cdot x:x\in V\}$}.
However, most of the existing data processing algorithms were designed for data in a Hilbert space, not an orbit space.
{Luckily, as the following section demonstrates,} a bilipschitz embedding of the orbit space into a Hilbert space will enable one to use such algorithms.
{This motivates the study of whether a given orbit space can be embedded into a Hilbert space by a bilipschitz map.
In this paper, we identify several necessary and sufficient conditions for such embeddings, and in some cases, we explicitly construct bilipschitz embeddings that minimally distort the quotient distance.}

Section~\ref{sec.metric quotient} defines the \textit{metric quotient} $V\dslash G$, whose points are the topological closures of $G$-orbits in $V$.
This matches the intuition that salient features are continuous functions of the data; also, while $V/G$ might only be a pseudometric space, $V\dslash G$ is always a metric space.
Section~\ref{sec:homogeneous extension} shows how a bilipschitz invariant on the sphere can be extended to a bilipschitz invariant on the entire space.
In Section~\ref{sec.nondiff bilipschitz}, we show that every bilipschitz invariant map is not differentiable.
This represents a significant departure from classical invariant theory~\cite{Hilbert:93,Olver:99,Sturmfels:08,BandeiraBKPWW:17}, which studies polynomial (hence differentiable) invariants.
Section~\ref{sec.bilipschitz poly} then uses the extension from Section~\ref{sec:homogeneous extension} to construct bilipschitz invariants for finite $G\leq O(d)$ from bilipschitz polynomial invariants on the sphere, which in turn exist precisely when $G$ acts freely on the sphere (hence rarely).
In Section~\ref{sec.minimum distortion}, we use the semidefinite program from \cite{LinialLR:95} to show that some of our extensions of polynomial invariants deliver the minimum possible distortion of $V\dslash G$ into a Hilbert space.
Finally, Sections~\ref{sec.permutation} and~\ref{sec.translation} estimate the Euclidean distortion of infinite-dimensional Hilbert spaces modulo permutation and translation groups, respectively.

\subsection{Related work}

The literature considers many ways of processing data in an orbit space.
In what follows, we discuss some relevant examples, which we broadly categorize according to their inspiration.

\subsubsection{Deep learning}

{First, we discuss how modern tools in machine learning currently process data in orbit spaces.
(In short, while there has been some effort to construct Lipschitz feature maps, there has been little work to date on bilipschitz feature maps.)}

Much of today's excitement in machine learning comes on the heels of \textit{deep neural networks} achieving super-human performance in various tasks.
For data with translation invariance, such as images and audio signals, it is empirically beneficial for the early layers of a neural network to exhibit convolutional structure~\cite{SimardSP:03,CiresanMGS:10,KrizhevskySH:12,GuEtal:18}.
Since convolutions are equivariant to translation, one might think of these early layers as isolating translation-equivariant features of the data.
The real-world success of \textit{convolutional neural networks} inspired the development of more principled versions of this fundamental idea.
For example, Mallat introduced the \textit{scattering transform}~\cite{Mallat:12,BrunaM:11,BrunaM:13,Waldspurger:17}, which iteratively alternates between taking a wavelet decomposition and a pointwise absolute value.
The scattering transform is invariant to translation and stable to diffeomorphism, and it has since been generalized to other settings~\cite{GaoWH:19,PerlmutterGWH:19,ChewHKNPSVW:22}.

A modern goal in this vein is to develop deep learning tools to accommodate non-Euclidean data, such as data in orbit spaces.
This is the charge of \textit{geometric deep learning}~\cite{BronsteinBLSV:17,BronsteinBCV:21}.
As an example, consider the task of classifying molecules according to whether they are harmful to humans.
We can represent the molecule as a graph, with vertices representing atoms and edges representing bonds, and then we can train a \textit{graph neural network} for this classification task.
By design, graph neural networks are invariant to relabeling the vertices of the input graph.
As such, they either fail to separate certain pairs of non-isomorphic graphs, or they are slow (assuming the graph isomorphism problem is hard).
The standard message-passing graph neural networks distinguish isomorphism classes of graphs as well as the Weisfeiler--Leman test~\cite{WeisfeilerL:68,XuHLJ:19,Sato:20,HuangV:21}, while other approaches achieve better separation~\cite{MorrisRFHLRG:19,BokerLHVM:23,HordanAGD:23}.

\subsubsection{Invariant theory}

One may express any group-invariant classifier as the composition of a group-invariant feature map with a classifier on the feature domain.
The feature map is most expressive if it separates all pairs of distinct orbits.
Under mild conditions on the group (e.g., if the group is finite), Hilbert~\cite{Hilbert:93} showed that there exists a polynomial map into finite-dimensional space that separates the orbits; in fact, one may take the coordinate functions to be generators of the algebra of invariant polynomials.
Modern work in invariant theory constructs smaller separating sets of polynomials with bounds on their degrees~\cite{Dufresne:08,DerksenK:15,Domokos:17,Domokos:22}.
{(Warning:\ polynomial invariants are considered to be \textit{separating} if they separate orbits as well as the entire ring of polynomial invariants, which might fail to distinguish certain orbits if the group is not compact.)}
Such polynomial maps play a key role in both invariant and equivariant machine learning~\cite{VillarHSYB:21,VillarYHBD:23,BlumSmithV:22}.

Polynomial invariants are also used to solve the \textit{orbit recovery problem}.
For this problem, there is a known group $G$ acting linearly on a vector space $V$, and the task is to recover the orbit $G\cdot x$ of an unknown point $x\in V$ from data of the form $\{y_i=g_ix+e_i\}_{i=1}^n$, where each $g_i\in G$ and $e_i\in V$ is drawn independently at random.
(In some settings, this problem is also known as \textit{multi-reference alignment} and can be viewed as a sub-problem of \textit{cryo-electron microscopy}.)
Much of the recent work on this problem applies the \textit{method of moments}, where one uses the data to estimate the moment $m_k(x):=\mathbb{E}_{g\sim\operatorname{Unif}(G)}(gx)^{\otimes k}$ for various $k$, and then estimates the orbit $G\cdot x$ from these moments~\cite{BandeiraBKPWW:17,PerryWBRS:19,BendoryELS:22,BendoryE:22,EdidinS:23}.
Notably, the coordinates of $m_k(x)$ are $G$-invariant polynomials of $x$.

Polynomial invariants have been investigated for well over a century, and the large body of existing work makes them particularly nice to interact with in theory.
Unfortunately, they can be rather finicky in practice.
As an example, consider the \textit{Vi\`{e}te map}, which sends any tuple $\{a_i\}_{i=1}^n$ of complex numbers to the coefficients of the corresponding degree-$n$ monic polynomial $\prod_{i=1}^n(z-a_i)$.
This defines a homeomorphism $\mathbb{C}^n/S_n\to\mathbb{C}^n$ with the elementary symmetric polynomials as coordinate functions~\cite{Whitney:72}.
While the continuity of the inverse map is most cumbersome to demonstrate where the polynomial has repeated roots, the inverse map is also unstable elsewhere, perhaps most famously at \textit{Wilkinson's polynomial} $\prod_{k=1}^{20}(z-k)$~\cite{Wilkinson:84}.
For this polynomial, perturbing the coefficient of $z^{19}$ by machine precision produces a sizable error in several of the roots.
Considering we are motivated by computational applications, this suggests that we pursue a more stable family of invariants.

\subsubsection{Phase retrieval}

Consider the problem of recovering a vector $x\in\mathbb{C}^d$ from data of the form $\{y_i=|\langle a_i,x\rangle|\}_{i=1}^n$.
This is known as \textit{(generalized) phase retrieval}.
Note that the map $x\mapsto \{|\langle a_i,x\rangle|\}_{i=1}^n$ is $G$-invariant, where $G=\{\lambda I_d:\lambda\in\mathbb{C},|\lambda|=1\}$.
Over the last decade, researchers have investigated conditions under which this map separates orbits~\cite{BalanCE:06,FickusMNW:14,BodmannH:15,ConcaEHV:15,Vinzant:15} and in a stable way~\cite{BandeiraCMN:14,BalanW:15}.
In addition, there are now several algorithms to invert this invariant map using semidefinite programming~\cite{Candes:15,CandesSV:13,DemanetH:14,WaldspurgerAM:15,GrossKK:17}, combinatorics~\cite{AlexeevBFM:14,BandeiraCM:14}, and non-convex optimization~\cite{CandesLS:15,ChenC:17,SunQW:18}.

The theory that was developed to study this particular family of invariants has recently inspired the construction of new families of invariants for many more group actions.
For example, \cite{CahillCC:20,CahillCC:19} modified polynomial separating invariants for abelian groups in order to be Lipschitz, \cite{BalanHS:22} identified bilipschitz permutation invariants based on sorting, and \cite{DymG:22,AmirGARD:23} discovered semialgebraic separating invariants for a variety of groups.
Finally, \cite{CahillIMP:22,MixonP:22,MixonQ:22} introduced and studied \textit{max filtering}, which directly generalizes phase retrieval to produce easy-to-compute semialgebraic separating invariants for any closed group.
Furthermore, max filtering invariants are frequently bilipschitz in the quotient metric, thereby avoiding the stability problem exemplified by Wilkinson's polynomial.

\subsubsection{Metric embeddings}

Consider the following fundamental question:
Given metric spaces $M$ and $N$, does there exist a bilipschitz map $M\to N$, i.e., a \textit{metric embedding} of $M$ into $N$?
This is an active area of research; see \cite{Naor:15} for example.
In the cases where a metric embedding exists, one seeks metric embeddings of minimal \textit{distortion}, that is, the quotient of upper and lower Lipschitz bounds.
If $M$ is finite and $N$ is $|M|$-dimensional Euclidean space, the distortion-minimizing embeddings can be determined by semidefinite programming~\cite{LinialLR:95,LinialM:00}.
Building on ideas from~\cite{KhotN:06}, Eriksson-Bique~\cite{ErikssonBique:18} constructed a metric embedding of $\mathbb{R}^d/G$ for any finite $G\leq O(d)$ into a Euclidean space.
{(This construction was later generalized in~\cite{Zolotov:19}.)}
Furthermore, the dimension of the target space and the distortion of the map are both bounded by functions of $d$.
Note that any metric embedding $\mathbb{R}^d/G\to\mathbb{R}^n$ determines a $G$-invariant map $\mathbb{R}^d\to\mathbb{R}^n$ that separates $G$-orbits in a stable way.
To date, Eriksson-Bique's map and max filtering are the only known methods of constructing bilipschitz invariants for arbitrary finite subgroups of $O(d)$.\footnote{{In the time since the original version of this paper was posted on the arXiv, max filtering was generalized in~\cite{BalanT:23} to a larger family of bilipschitz invariants known as \textit{coorbit embeddings}.}}

\section{Bilipschitz maps and their applications}

Given a map $f\colon X\to Y$ between metric spaces $X$ and $Y$, then provided $X$ has at least two points, we may take $\alpha,\beta\in[0,\infty]$ to be the largest and smallest constants (respectively) such that
\[
\alpha\cdot d_X(x,x')
\leq d_Y\big(f(x),f(x')\big)
\leq \beta\cdot d_X(x,x')
\qquad
\forall x,x'\in X.
\]
These are the (optimal) lower and upper Lipschitz bounds, respectively.
(Notice that $\alpha$ and $\beta$ are not well defined when $X$ has fewer than two points.)
The \textbf{distortion} of $f$ is
\[
\operatorname{dist}(f):=\frac{\beta}{\alpha},
\]
with the convention that division by zero is infinite.
A map with finite upper Lipschitz bound is called \textbf{Lipschitz}, a map with positive lower Lipschitz bound is called \textbf{lower Lipschitz}, and a map with finite distortion is called \textbf{bilipschitz}.
Observe that Lipschitz functions are necessarily uniformly continuous, and lower Lipschitz functions are necessarily invertible on their range with uniformly continuous inverse.

Bilipschitz maps are particularly useful in the context of data science.
Indeed, while many data science tools have been developed for data in Euclidean space, in many cases, data naturally arises in other metric spaces.
Given a metric space $X$ and a map $f\colon X\to\mathbb{R}^d$, one may pull back these Euclidean tools through $f$ to accommodate data in $X$.
Furthermore, we can estimate the quality of this pullback in terms of the bilipschitz bounds of $f$.
Three examples of this phenomenon follow.

\begin{example}[Nearest neighbor search]
Fix data $\{x_i\}_{i\in I}$ in a metric space $X$ and a parameter $\lambda\geq1$.
The \textbf{$\lambda$-approximate nearest neighbor problem} takes as input some $x\in X$ and outputs $j\in I$ such that
\[
d_X(x,x_j)
\leq \lambda\cdot\min_{i\in I}d_X(x,x_i).
\]
To date, many algorithms solve this in the case where $X$ is Euclidean~\cite{AndoniIR:18}.
When $X$ is non-Euclidean, we may combine such a black box with a map $f\colon X\to\mathbb{R}^d$ of distortion $c$ to solve the $c\lambda$-approximate nearest neighbor problem in $X$.
Indeed, find $j\in I$ such that
\[
\|f(x)-f(x_j)\|
\leq \lambda\cdot\min_{i\in I}\|f(x)-f(x_i)\|.
\]
Then
\[
d_X(x,x_j)
\leq\frac{1}{\alpha}\cdot \|f(x)-f(x_j)\|
\leq\frac{\lambda}{\alpha}\cdot\min_{i\in I}\|f(x)-f(x_i)\|
\leq c\lambda\cdot\min_{i\in I} d_X(x,x_i).
\]
\end{example}

\begin{example}[Clustering]
Given data $\{x_i\}_{i\in I}$ in a metric space $X$ and a parameter $k\in\mathbb{N}$, we seek to partition the data into $k$ clusters.
There are many clustering objectives to choose from, but the most popular choice when $X=\mathbb{R}^d$ is the \textbf{$k$-means objective}:
\[
\text{minimize}
\quad
\sum_{t=1}^k\frac{1}{|C_t|}\sum_{i\in C_t}\sum_{j\in C_t}d_X(x_i,x_j)^2
\quad
\text{subject to}
\quad 
C_1\sqcup\cdots \sqcup C_k=I,
\quad
C_1,\ldots,C_k\neq\emptyset.
\]
This objective is popular in Euclidean space since Lloyd's algorithm is fast and works well in practice.
In addition, $k$-means++ delivers a $O(\log k)$-competitive solution to the $k$-means problem~\cite{ArthurV:07}.
When $X$ is non-Euclidean, we may combine a $\lambda$-competitive solver for $\mathbb{R}^d$ with a map $f\colon X\to\mathbb{R}^d$ of distortion $c$ to obtain a $c^2\lambda$-competitive solver for $X$.
Indeed, suppose $\{C'_t\}_{t=1}^k$ is a $\lambda$-competitive $k$-means clustering for $\{f(x_i)\}_{i\in I}$.
Then
\begin{align*}
\sum_{t=1}^k\frac{1}{|C'_t|}\sum_{i\in C'_t}\sum_{j\in C'_t}d_X(x_i,x_j)^2
&\leq \sum_{t=1}^k\frac{1}{|C'_t|}\sum_{i\in C'_t}\sum_{j\in C'_t}\frac{1}{\alpha^2}\cdot\|f(x_i)-f(x_j)\|^2\\
&\leq \frac{\lambda}{\alpha^2} \min_{\substack{C_1\sqcup\ldots\sqcup C_k=I\\C_1,\ldots,C_k\neq\emptyset}}\sum_{t=1}^k\frac{1}{|C_t|}\sum_{i\in C_t}\sum_{j\in C_t}\|f(x_i)-f(x_j)\|^2\\
&\leq c^2\lambda \min_{\substack{C_1\sqcup\ldots\sqcup C_k=I\\C_1,\ldots,C_k\neq\emptyset}}\sum_{t=1}^k\frac{1}{|C_t|}\sum_{i\in C_t}\sum_{j\in C_t}d_X(x_i,x_j)^2.
\end{align*}
\end{example}

\begin{example}[Visualization]
To visualize data $\{x_i\}_{i=1}^n$ in a metric space $X$, one might apply \textbf{multidimensional scaling}~\cite{KruskalW:78} to represent the data as points in a Euclidean space of dimension $k\in\{2,3\}$.
The multidimensional scaling objective is
\[
\text{minimize}
\quad
\|Z^\top Z-g(D)\|_F
\quad
\text{subject to}
\quad
Z\in\mathbb{R}^{k\times n}.
\]
Here, $D_{ij}=d_X(x_i,x_j)^2$ and $g(D)=-\frac{1}{2}(I-\frac{1}{n}J)D(I-\frac{1}{n}J)$, where $I$ is the identity matrix and $J$ is the matrix of all ones.
The minimizer can be obtained by taking the eigenvalue decomposition $g(D)=\sum_i \lambda_i u_iu_i^\top$ with $\lambda_1\geq\cdots\geq\lambda_n$ and taking the $i$th row of $Z$ to be $\sqrt{(\lambda_i)_+}\cdot u_i^\top$.
Note that this requires one to compute the top eigenvalues and eigenvectors of an $n\times n$ matrix.
Meanwhile, if $X=\mathbb{R}^d$, this is equivalent to running principal component analysis on $\{x_i\}_{i=1}^n$, which only requires computing the top eigenvalues and eigenvectors of a $d\times d$ matrix, and so the runtime is linear in $n$.
This speedup is available to non-Euclidean $X$ given a map $f\colon X\to\mathbb{R}^d$ with bilipschitz bounds $\alpha$ and $\beta$.
In particular, let $Y\in\mathbb{R}^{d\times n}$ have $i$th column $f(x_i)-\mu$, where $\mu=\frac{1}{n}\sum_if(x_i)$. 
Then
\[
\min_{Z\in\mathbb{R}^{k\in n}}\|Z^\top Z-Y^\top Y\|_F
\leq \min_{Z\in\mathbb{R}^{k\in n}}\|Z^\top Z-g(D)\|_F+\|g(D)-Y^\top Y\|_F.
\]
Furthermore, taking $E\in\mathbb{R}^{n\times n}$ with $E_{ij}=\|f(x_i)-f(x_j)\|^2$, then $g(E)=Y^\top Y$.
Since $-2g$ is an orthogonal projection in the space of symmetric matrices, we have
\[
\|g(D)-Y^\top Y\|_F
\leq \frac{1}{2}\|E-D\|_F
\leq\frac{1}{2}\cdot\max\Big\{|\alpha^2-1|,|\beta^2-1|\Big\}\cdot\|D\|_F.
\]
As such, the faster solution introduces an additive error that is smaller when $\alpha$ and $\beta$ are both closer to $1$.
This approach was used in \cite{CahillIMP:22} to visualize how the shapes of U.S.\ Congressional districts are distributed.
\end{example}

\section{The metric quotient}
\label{sec.metric quotient}

In this section, we identify an appropriate notion of quotient for metric spaces.
Given a group $G$ acting on a metric space $X$, one may consider the set of orbits
\[
X/G
:=\{G\cdot x:x\in X\}.
\]
Write $x\sim y$ if $G\cdot x=G\cdot y$.
This quotient is endowed with a pseudometric $d_{X/G}$ defined by
\[
d_{X/G}(G\cdot x,G\cdot y)
=\inf \sum_{i=1}^n d_X(p_i,q_i),
\]
where the infimum is taken over all $n\in\mathbb{N}$ and $p_1,q_1,\ldots,p_n,q_n\in X$ such that $x\sim p_1$, $q_i\sim p_{i+1}$ for each $i\in\{1,\ldots,n-1\}$, and $q_n\sim y$. 

\begin{example}
Suppose $X=\mathbb{R}^2$ with Euclidean distance and $G=\{\left[\begin{smallmatrix}t&0\\0&1/t\end{smallmatrix}\right]:t\neq0\}$.
Then the members of $X/G$ are given by
\begin{align*}
\text{the origin:} \qquad & G\cdot(0,0)=\{(0,0)\} \\
\text{the $x$-axis:} \qquad & G\cdot(1,0)=\{(x,0):x\neq0\} \\
\text{the $y$-axis:} \qquad & G\cdot(0,1)=\{(0,y):y\neq 0\} \\
\text{hyperbolas:} \qquad & G\cdot(c,1)=\{(x,y):xy=c\}, \quad c\neq 0.
\end{align*}
Furthermore, $d_{X/G}\equiv 0$ since both axes are arbitrarily close to the origin and each hyperbola is arbitrarily close to both axes.
\end{example}

In this paper, we are primarily interested in cases where $G$ acts by isometries on $X$, in which case an easy argument akin to {the proof of Theorem~4 in}~\cite{Himmelberg:68} gives that
\[
d_{X/G}(G\cdot x,G\cdot y)
=\inf_{\substack{p\sim x\\q\sim y}}d_X(p,q)
=\inf_{p\sim x}d_X(p,y)
=\inf_{q\sim y}d_X(x,q).
\]

\begin{example}
Suppose $X=\mathbb{R}^2$ with Euclidean distance and $G=SO(2)$.
Then the members of $X/G$ are given by
\begin{align*}
\text{the origin:} \qquad & G\cdot(0,0)=\{(0,0)\} \\
\text{circles:} \qquad & G\cdot(r,0)=\{(r\cos\theta,r\sin\theta):\theta\in[0,2\pi)\}, \quad r>0.
\end{align*}
Furthermore, the reverse triangle inequality implies $d_{X/G}(G\cdot x,G\cdot y)=\big|\|x\|-\|y\|\big|$.
Since $d_{X/G}$ is nondegenerate, it defines a metric on $X/G$.
\end{example}

\begin{example}
\label{ex.ell2 with permutations}
Suppose $X$ is the Hilbert space $\ell^2(\mathbb{N};\mathbb{R})$ of square summable real-valued sequences and $G$ is the group of bijections $\mathbb{N}\to\mathbb{N}$.
Then $d_{X/G}$ does not define a metric on $X/G$.
Indeed, suppose $x\in X$ is entrywise nonzero, and observe
\begin{align*}
x:=(x_1,x_2,x_3,x_4,\ldots)
\sim(x_2,x_1,x_3,x_4,\ldots)
\sim(x_3,x_1,x_2,x_4,\ldots)
\sim(x_4,x_1,x_2,x_3,\ldots).
\end{align*}
We may continue in this way to construct a sequence in $G\cdot x$ that converges to the right shift $y:=(0,x_1,x_2,x_3,x_4,\ldots)$, but $y\not\in G\cdot x$ since $x$ is entrywise nonzero.
Thus, $G\cdot x\neq G\cdot y$, but $d_{X/G}(G\cdot x,G\cdot y)=0$.
\end{example}

Recall that a pseudometric space determines a metric space by identifying points of distance zero.
This motivates the definition of the \textbf{metric quotient}:
\[
X\dslash G:=\{[x]:x\in X\},
\qquad
[x]:=\{y\in X:d_{X/G}(G\cdot x,G\cdot y)=0\},
\]
\[
d_{X\dslash G}([x],[y]):=d_{X/G}(G\cdot x,G\cdot y).
\]
One may show that $d_{X\dslash G}$ is a metric on $X\dslash G$.
Here, we use $\dslash$ to signify that we are taking two quotients: we mod out by $G$, and then by the zero set of $d_{X/G}$.
This notation has been used previously in the literature to denote the \textit{geometric invariant theory quotient}~\cite{MumfordFK:94}, {a notion that directly inspired the authors' definition of metric quotient.}
A similar quotient is defined in \cite{Weaver:99} for complete metric spaces modulo general equivalence classes.
This quotient uses the same approach of identifying points with pseudometric distance zero, but then takes the metric completion of the result.
In the special case where $X$ is complete and $G$ acts on $X$ by isometries, one can show that $X\dslash G$ is already complete, and so these notions of quotient coincide.

The metric quotient has a categorical interpretation.
For this, we view $X$ as an object in some concrete category $C$.
For a group $G$ acting on $X$, a morphism $\pi\colon X \to Y$ in $C$ is called a \textbf{categorical quotient} if
\begin{itemize}
\item[(C1)]
$\pi$ is $G$-invariant, and
\item[(C2)]
every $G$-invariant morphism $X \to Z$ uniquely factors through $\pi$.
\end{itemize}
The categorical quotient separates the $G$-orbits in $X$ as well as possible for a $G$-invariant morphism in $C$.
Categorical quotients are unique up to canonical isomorphism.

\begin{lemma}
\label{lem.metric quotient is categorical}
Consider any set $\Omega$ of functions $[0,\infty]\to[0,\infty]$ that satisfies the following:
\begin{itemize}
\item[(i)]
the identity function belongs to $\Omega$, 
\item[(ii)]
$\Omega$ is closed under composition, and
\item[(iii)]
every $\omega\in\Omega$ is weakly increasing, upper semicontinuous, and vanishing at zero.
\end{itemize}
Then $\Omega$ determines a category $C$ whose objects are all metric spaces and whose morphisms are all functions $f\colon Y\to Z$ for which there exists $\omega\in\Omega$ such that
\[
d_Z(f(y),f(y'))
\leq\omega(d_Y(y,y'))
\qquad
\forall y,y'\in Y,
\]
i.e., $f$ admits $\omega$ as a modulus of continuity.
Furthermore, {for any group $G$ acting by isometries on $X$,} the map $X\to X\dslash G$ defined by $x\mapsto[x]$ is a categorical quotient in $C$.
\end{lemma}

For example, the metric quotient is a categorical quotient for each category of metric spaces with one of the following choices of morphisms:
\begin{align*}
\text{metric maps:} \qquad & \Omega = \{ t\mapsto t \}\\
\text{Lipschitz maps:} \qquad & \Omega = \{ t\mapsto ct : c\geq0 \}\\
\text{H\"{o}lder continuous maps:} \qquad & \Omega = \{ t\mapsto ct^\alpha : c\geq0,~ \alpha>0 \}\\
\text{uniformly continuous maps:} \qquad & \Omega = \{ \text{weakly increasing, upper semicontinuous}\\
&\qquad~~ \text{functions $[0,\infty]\to[0,\infty]$ that vanish at zero} \}.
\end{align*}
In this paper, we are primarily interested in the Lipschitz category.
{In this category, a modification of the following proof shows that $X\to X\dslash G$ is a categorical quotient for \textit{any} group $G$ acting on $X$, even if not by isometries.}

\begin{proof}[Proof of Lemma~\ref{lem.metric quotient is categorical}]
First, $C$ is a category by (i) and (ii).
Indeed, the identity map on any metric space admits the identity function $\omega_{\operatorname{id}}\in\Omega$ as a modulus of continuity.
Also, if $f\colon Y\to Y'$ and $g\colon Y'\to Y''$ have respective moduli of continuity $\omega_f,\omega_g\in\Omega$, then $g\circ f$ has modulus of continuity $\omega_g\circ \omega_f\in\Omega$.

To show that $\pi\colon X\to X\dslash G$ defined by $\pi(x)=[x]$ is a categorical quotient for $C$, we first verify that $\pi$ is a morphism in $C$:
\[
d_{X\dslash G}(\pi(x),\pi(x'))
=d_{X/G}(G\cdot x,G\cdot x')
\leq d_X(x,x')
= \omega_{\operatorname{id}}(d_X(x,x')).
\]
Next, we verify (C1).
Indeed, for every $x\in X$ and $g\in G$, we have $d_{X/G}(G\cdot x,G\cdot gx)=0$, and so $gx\in[x]$, i.e., $\pi(gx)=[gx]=[x]=\pi(x)$.
Thus, $\pi$ is $G$-invariant.
Finally, we verify (C2).
Consider any $G$-invariant morphism $f\colon X\to Z$, with modulus of continuity $\omega_f\in \Omega$.
Then for any $x,x'\in X$, it holds that
\begin{equation}
\label{eq.repeated logic}
\inf_{\substack{p\sim x\\q\sim x'}}d_Z(f(p),f(q))
\leq \inf_{\substack{p\sim x\\q\sim x'}}\omega_f(d_X(p,q))
= \omega_f\Big(\inf_{\substack{p\sim x\\q\sim x'}}d_X(p,q)\Big)
=\omega_f(d_{X\dslash G}([x],[x'])),
\end{equation}
where the second step follows from the assumptions in (iii) that $\omega_f$ is weakly increasing and upper semicontinuous.
Now suppose $[x]=[x']$.
Then we may apply \eqref{eq.repeated logic} to get
\[
d_Z(f(x),f(x'))
=\inf_{\substack{p\sim x\\q\sim x'}}d_Z(f(p),f(q))
\leq\omega_f(d_{X\dslash G}([x],[x']))
=0,
\]
where the last step follows from the assumption in (iii) that $\omega_f$ is vanishing at zero.
Thus, $f$ is constant on the level sets of $\pi$, which uniquely determines $f^\downarrow\colon X\dslash G\to Z$ such that $f=f^\downarrow\circ \pi$.
It remains to show that $f^\downarrow$ is a morphism in $C$.
To this end, \eqref{eq.repeated logic} gives
\[
d_Z(f^\downarrow([x]),f^\downarrow([x']))
=\inf_{\substack{p\sim x\\q\sim x'}}d_Z(f(p),f(q))
\leq \omega_f(d_{X\dslash G}([x],[x'])),
\]
as desired.
\end{proof}

\begin{lemma}
\label{lem.metric quotient orbit closure}
Suppose $G$ acts on $X$ by isometries.
Then $[x]=\overline{G\cdot x}$ for each $x\in X$.
In particular, $X\dslash G=X/G$ if and only if the orbits of $G$ in $X$ are closed.
\end{lemma}

It is worth noting some sufficient conditions for the orbits of $G$ to be closed in $X$.
First, the orbits are closed whenever $G$ is finite.
Next, if $X$ is a normed vector space and $G$ is a subgroup of linear isometries, then the orbits of $G$ are closed in $X$ whenever $G$ is compact in the strong operator topology.
As an example in which $G$ acts on $X$ by isometries but $X\dslash G\neq X/G$, one may take $X=\mathbb{R}^2$ and $G\leq O(2)$ to be the group of rotations by rational multiples of $\pi$.
Indeed, in this example, the orbit of any unit vector is a dense but proper subset of the unit circle.

\begin{proof}[Proof of Lemma~\ref{lem.metric quotient orbit closure}]
Fix $x\in X$.
First, we show $[x]\subseteq\overline{G\cdot x}$.
Suppose $y\in[x]$.
Then
\[
\inf_{p\sim x}d_X(p,y)
=d_{X/G}(G\cdot x,G\cdot y)
=0,
\]
and so $y\in \overline{G\cdot x}$.
Next, $G\cdot x\subseteq[x]$, since $y\in G\cdot x$ implies $d_{X/G}(G\cdot x,G\cdot y)=0$, meaning $y\in [x]$.
Overall,
\[
G\cdot x
\subseteq[x]
\subseteq\overline{G\cdot x}.
\]
Thus, to prove $[x]=\overline{G\cdot x}$, it suffices to show that $[x]$ is closed.
To this end, consider $f\colon X\to\mathbb{R}$ defined by $f(y)=d_{X/G}(G\cdot x,G\cdot y)$.
We will show that $f$ is continuous, which in turn implies that $[x]=f^{-1}(\{0\})$ is closed, as desired.

Given $y\in X$ and $\epsilon>0$, put $\delta:=\frac{\epsilon}{2}$.
Then for $z\in X$ such that $d_X(y,z)<\delta$, we have
\[
f(z)
=\inf_{p\sim x}d_X(p,z)
\leq\inf_{p\sim x}\Big(d_X(p,y)+d_X(y,z)\Big)
\leq \inf_{p\sim x}d_X(p,y)+\delta
=f(y)+\delta,
\]
and similarly,
\[
f(z)
=\inf_{p\sim x}d_X(p,z)
\geq\inf_{p\sim x}\Big(d_X(p,y)-d_X(y,z)\Big)
\geq \inf_{p\sim x}d_X(p,y)-\delta
=f(y)-\delta.
\]
Combining these estimates gives $|f(z)-f(y)|\leq \delta<\epsilon$, as desired.
\end{proof}

Suppose $G$ acts on $X$ by isometries.
A continuous $G$-invariant map $f\colon X\to Y$ is constant on each orbit $G\cdot x$, and by continuity, it is also constant on the closure $[x]=\overline{G\cdot x}$.
In particular, $f$ factors through $f^\downarrow\colon X\dslash G\to Y$ defined by $f^\downarrow([x])=f(p)$ for $p\in[x]$:
\begin{center}
\begin{tikzcd}[row sep=large]
X \arrow[d, "\pi"'] \arrow[r, "f"] & Y\\
X\dslash G \arrow[ur, "f^\downarrow"']
\end{tikzcd}
\end{center}
Throughout, we reserve the superscript downarrow to denote this induced factor map over the metric quotient.
Also, we write $X/G$ instead of $X\dslash G$ when they coincide, but in such cases, we typically write $[x]$ instead of $G\cdot x$ for brevity.
In the same spirit, we write $d(\cdot,\cdot)$ instead of $d_{X\dslash G}(\cdot,\cdot)$ in the sequel.

This paper is primarily concerned with real Hilbert spaces modulo subgroups of the corresponding orthogonal group.
Unless stated otherwise, $V$ and $W$ denote nontrivial real Hilbert spaces, $S(V)$ is the unit sphere in $V$, and $G$ is a subgroup of the orthogonal group $O(V)$.
Whenever we have a complex inner product, we conjugate on the left.
Notably, every complex Hilbert space $V'$ is isometric to a real Hilbert space $V''$ with inner product $\operatorname{Re}\langle\cdot,\cdot\rangle_{V'}$, and furthermore, $U(V')\leq O(V'')$.
{For this reason, we do not lose generality by restricting our attention to real Hilbert spaces, though as we will see, some important examples are more naturally expressed in terms of complex Hilbert spaces.}

The \textbf{max filtering map} $\llangle\cdot,\cdot\rrangle\colon V\dslash G\times V\dslash G\to\mathbb{R}$ is defined by
\[
\llangle [x],[y]\rrangle
:=\sup_{\substack{p\in[x]\\q\in[y]}}\langle p,q\rangle.
\]
This naturally emerges in the cross term when expanding the square of the quotient metric:
\[
d([x],[y])^2
=\|x\|^2-2\llangle [x],[y]\rrangle+\|y\|^2.
\]
For a complex Hilbert space and a group of unitaries, max filtering is defined by passing to the underlying real Hilbert space:
\[
\llangle [x],[y]\rrangle
=\sup_{\substack{p\in[x]\\q\in[y]}}\operatorname{Re}\langle p,q\rangle.
\]
Max filtering was introduced in~\cite{CahillIMP:22} to construct bilipschitz invariant maps for arbitrary finite groups $G\leq O(d)$:

\begin{proposition}[Theorem~18 in~\cite{CahillIMP:22}, cf.\ Theorem~18 in~\cite{MixonQ:22}]
\label{prop.max filtering bilipschitz}
For every finite $G\leq O(d)$, there exist $n\in\mathbb{N}$ and $z_1,\ldots,z_n\in\mathbb{R}^d$ such that $[x]\mapsto \{\llangle [z_i],[x]\rrangle\}_{i=1}^n$ is bilipschitz.
\end{proposition}

\section{Homogeneous extension}
\label{sec:homogeneous extension}

We motivate this section with an example:

\begin{example}
\label{ex.finite dimensional real phase retrieval}
Take $V:=\mathbb{R}^d$ and $G:=\{\pm\operatorname{id}\}\leq O(V)$, and consider the $G$-invariant map $f\colon x\mapsto xx^\top$.
The induced map $f^\downarrow\colon \mathbb{R}^d/G\to\mathbb{R}^{d\times d}$ is injective since {the column space and trace of the matrix $f^\downarrow([x])$ are given by}
\[
\operatorname{im}f^\downarrow([x])=\operatorname{span}\{x\}
\qquad
\text{and}
\qquad
\operatorname{tr}f^\downarrow([x])=\|x\|^2, 
\]
which together determine $\{\pm x\}=[x]$.
However, Figure~\ref{fig.in out distance} (left) illustrates that $f^\downarrow$ is not lower Lipschitz.
As we will see in Theorem~\ref{thm.bilipschitz requires nondifferentiability}(b), this is an artifact of the differentiability of $f$ at the origin.
In fact, we can say more: for any unit vector $u\in\mathbb{R}^d$ and $c\geq0$, we have
\[
d([cu],[0])
=\|cu\|
=c,
\qquad
\|f^\downarrow([cu])-f^\downarrow([0])\|_F
=\|c^2uu^\top\|_F
=c^2.
\]
The limit $c\to0$ establishes the lower Lipschitz bound $\alpha=0$, and $c\to\infty$ establishes $\beta=\infty$.
While $f^\downarrow$ fails to be bilipschitz at points ``far'' from $S(\mathbb{R}^d)/G$, Figure~\ref{fig.in out distance} (middle) indicates that $f^\downarrow$ is bilipschitz when restricted to $S(\mathbb{R}^d)/G$.
This suggests that $f$ is only problematic in the radial direction, which can be corrected by instead mapping $g\colon cu\mapsto cuu^\top$ for any unit vector $u\in\mathbb{R}^d$ and $c\geq0$.
This extends $f|_{S(\mathbb{R}^d)}$ to all of $\mathbb{R}^d$ in a radially isometric way:
\[
d([cu],[0])
=\|cu\|
=c,
\qquad
\|g^\downarrow([cu])-g^\downarrow([0])\|_F
=\|cuu^\top\|_F
=c.
\]
Furthermore, Figure~\ref{fig.in out distance} (right) indicates that this extension is bilipschitz on all of $\mathbb{R}^d/G$.
\end{example}

\begin{figure}
\begin{center}
\includegraphics[width=0.3\textwidth,trim={90 0 90 0},clip]{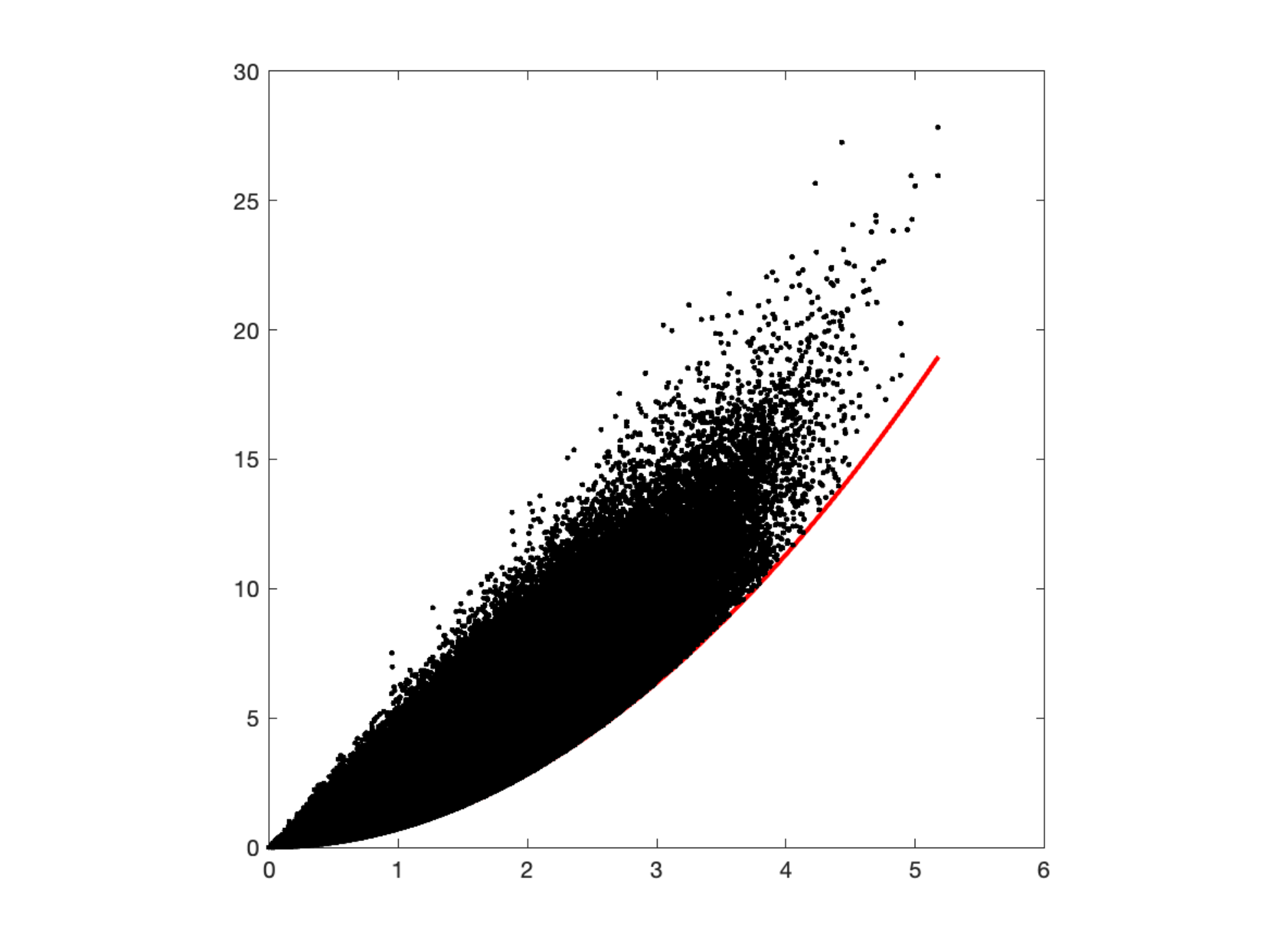}
\quad
\includegraphics[width=0.3\textwidth,trim={90 0 90 0},clip]{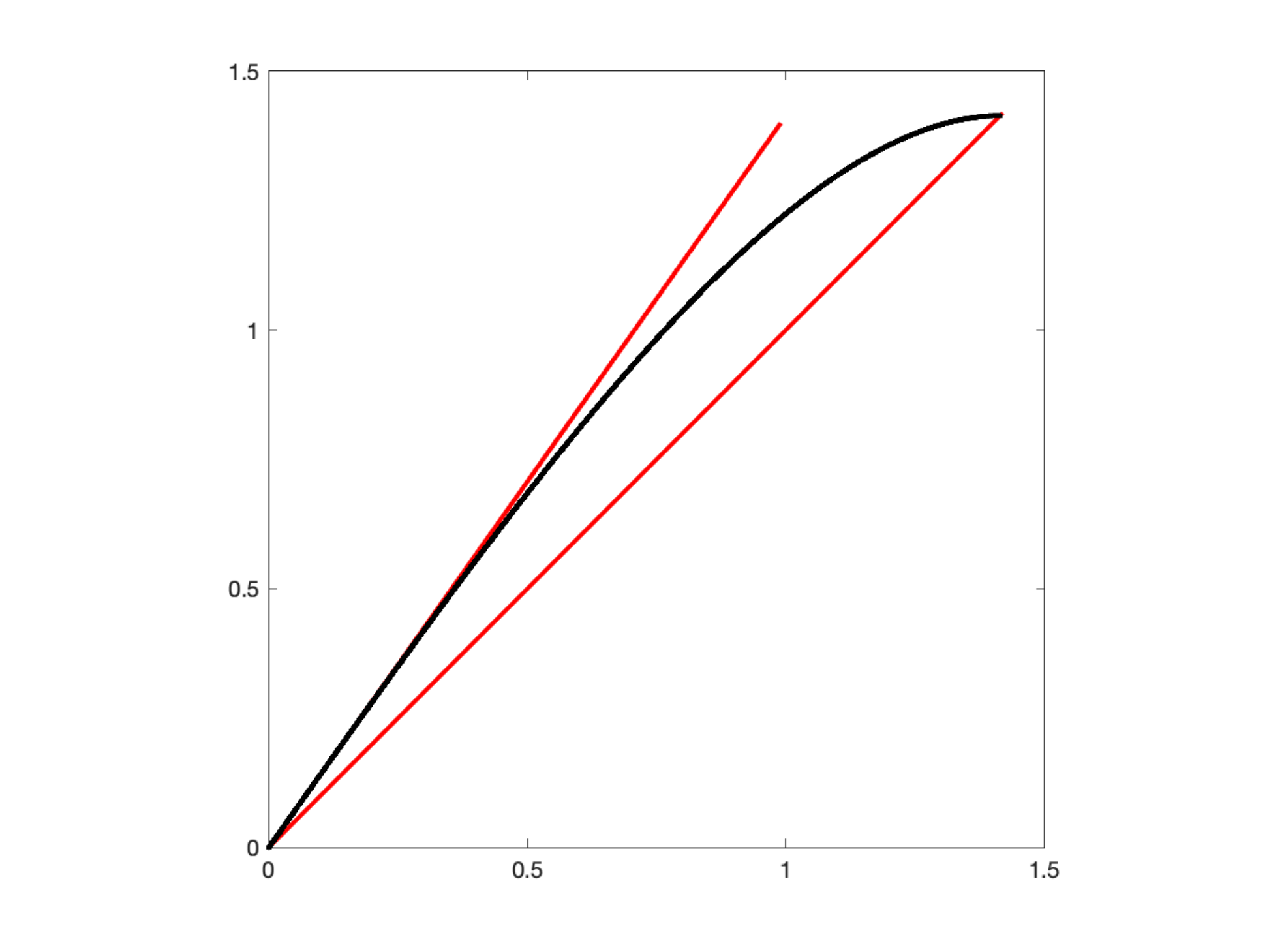}
\quad
\includegraphics[width=0.3\textwidth,trim={90 0 90 0},clip]{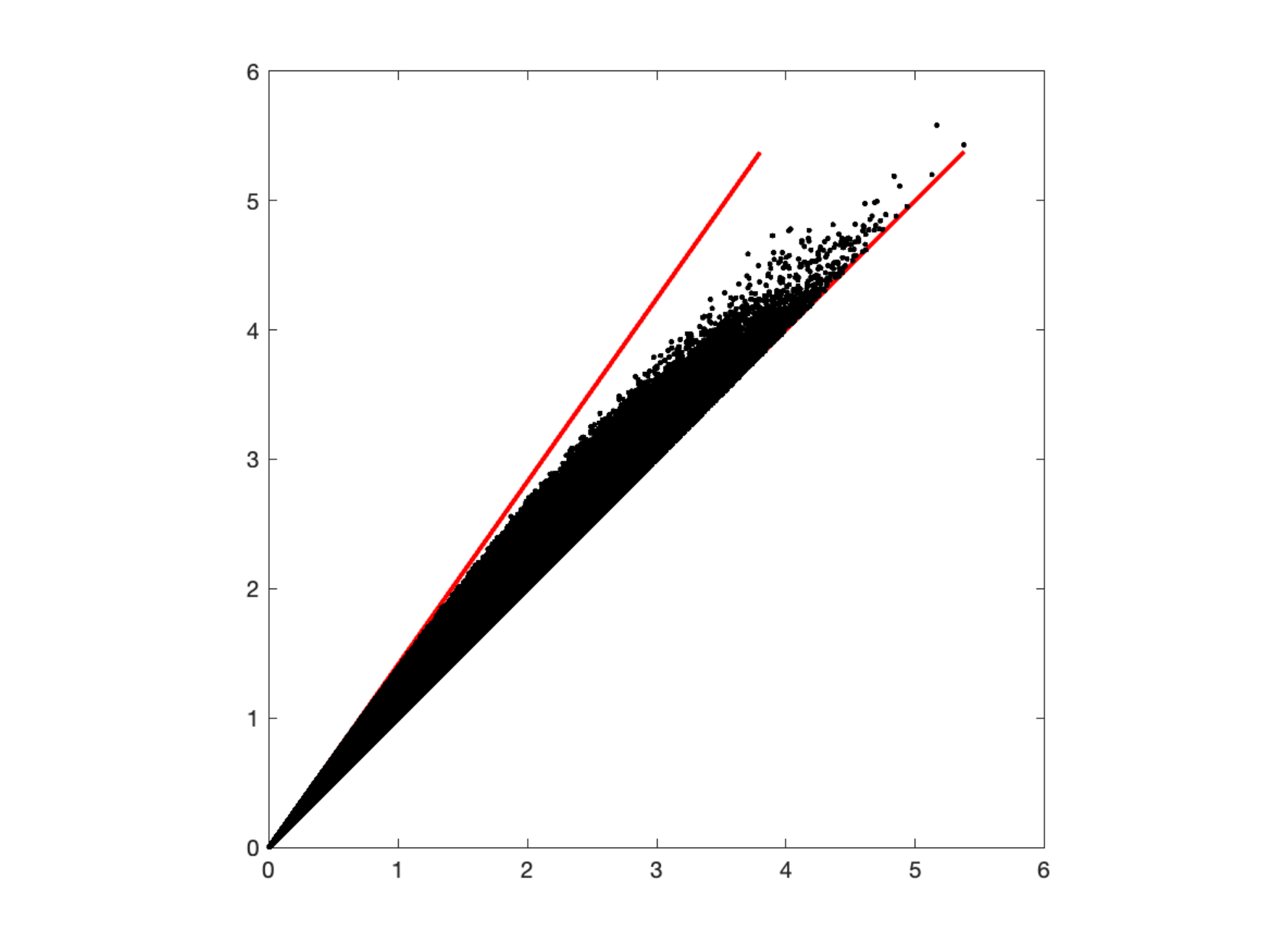}
\end{center}
\caption{\label{fig.in out distance}
Illustration of Example~\ref{ex.finite dimensional real phase retrieval}.
Here, we take $V:=\mathbb{R}^2$ and $G:=\{\pm\operatorname{id}\}\leq O(2)$.
The horizontal and vertical axes represent input and output distances, respectively.
\textbf{(left)}
Draw a million pairs of vectors $x,y\in\mathbb{R}^2$ with standard gaussian distribution and plot the output distance $\|xx^\top-yy^\top\|_F$ versus the input distance $d([x],[y])$.
One can show that if the input distance is $a>0$, then the output distance can take any value in $[a^2/\sqrt{2},\infty)$.
The lower bound is depicted in red.
Notably, the map $[x]\mapsto xx^\top$ is neither Lipschitz nor lower Lipschitz.
\textbf{(middle)}
Draw a million pairs of vectors $x,y\in\mathbb{R}^2$ uniformly from the unit circle and plot the output distance $\|xx^\top-yy^\top\|_F$ versus the input distance $d([x],[y])$.
One can show that that if the input distance is $a>0$, then the output distance resides in the interval $[a,\sqrt{2}a]$.
These bounds are depicted in red.
Notably, the map $[x]\mapsto xx^\top$ is bilipschitz when restricted to $S(\mathbb{R}^2)/G$.
\textbf{(right)}
Draw a million pairs of vectors $x,y\in\mathbb{R}^2$ with standard gaussian distribution and plot the output distance $\|\frac{1}{\|x\|}xx^\top-\frac{1}{\|y\|}yy^\top\|_F$ versus the input distance $d([x],[y])$.
One can show that that if the input distance is $a>0$, then the output distance resides in the interval $[a,\sqrt{2}a]$.
These bounds are depicted in red.
Notably, the map $[x]\mapsto \frac{1}{\|x\|}xx^\top$ is bilipschitz.
More generally, Theorem~\ref{thm.bilipschitz homogeneous extension} gives that the homogeneous extension of a bilipschitz map is bilipschitz.
}
\end{figure}

The above example motivates the following definition:

\begin{definition}
\label{def.homogeneous extension}
The \textbf{homogeneous extension} of $f\colon S(V)\dslash G\to S(W)$ is $f^\star\colon V\dslash G\to W$ defined by $f^\star([cu])=cf([u])$ for $u\in S(V)$ and $c\geq0$.
\end{definition}

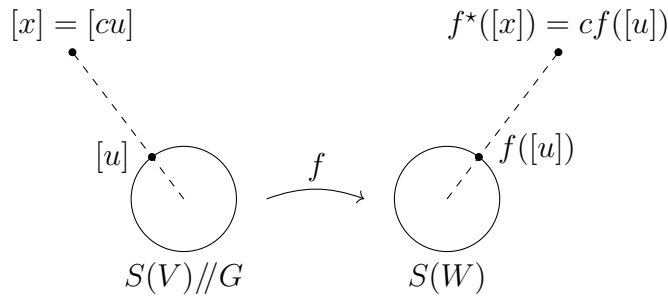
\begin{figure}
\begin{center}
\begin{tikzpicture}
\coordinate (u) at (-0.6042*0.7,0.79682*0.7);
\coordinate (x) at (-3.5*0.6042*0.7,3.5*0.7968*0.7);
\draw[dashed] (0,0) -- (x);
\draw[fill=none] (0,0) circle (1.0*0.7) node [yshift=-30] {$S(V)\dslash G$};
\draw[fill] (u) circle (1.3 pt) node [xshift=-15] {$[u]$};
\draw[fill] (x) circle (1.3 pt) node [above] {$[x]=[cu]$};
\draw [->] (1.1,0) to [out=20,in=160] (2.4,0) node[xshift=-18,yshift=12] {$f$};
\coordinate (gu) at (3.5+0.6042*0.7,0.79682*0.7);
\coordinate (gx) at (3.5+3.5*0.6042*0.7,3.5*0.7968*0.7);
\draw[dashed] (3.5,0) -- (gx);
\draw[fill=none] (3.5,0) circle (1.0*0.7) node [yshift=-30] {$S(W)$};
\draw[fill] (gu) circle (1.3 pt) node [xshift=22,yshift=3] {$f([u])$};
\draw[fill] (gx) circle (1.3 pt) node [above] {$f^\star([x])=cf([u])$};
\end{tikzpicture}
\end{center}
\caption{\label{fig.homogeneous extension}
Illustration of homogeneous extension.
See Definition~\ref{def.homogeneous extension}.}
\end{figure}

See Figure~\ref{fig.homogeneous extension} for an illustration.
A related notion appears in {the proof of Lemma~3.5 in}~\cite{ErikssonBique:18}.
It is routine to verify that the homogeneous extension is well defined.
We note that \cite{CahillCC:20} considers homogeneous extensions of functions with codomain $W$ (see Definition~2 in~\cite{CahillCC:20}).
By instead taking the codomain to be $S(W)$, the lower Lipschitz bound of $f$ quantifies the \textit{non-parallel property} (see Definition~1 in~\cite{CahillCC:20}), which in turn produces a lower Lipschitz bound of the homogeneous extension; this is a quantitative version of Proposition~1(b) in~\cite{CahillCC:20}.
In fact, we obtain \textit{optimal} bilipschitz bounds in Theorem~\ref{thm.bilipschitz homogeneous extension} below (cf.\ the upper Lipschitz bound given in Theorem~4 of~\cite{CahillCC:20} and the distortion bound given in Lemma~3.5 of~\cite{ErikssonBique:18}.).

\begin{lemma}[radial Pythagorean theorem]
\label{lem.metric extension}
For $u,v\in S(V)$ and $a,b\geq0$, it holds that
\[
d([au],[bv])^2
=(a-b)^2+ab\cdot d([u],[v])^2.
\]
\end{lemma}

\begin{proof}
Write $d([au],[bv])^2=a^2-2ab\llangle [u],[v]\rrangle +b^2$ and apply $2\llangle [u],[v]\rrangle=2-d([u],[v])^2$.
\end{proof}

When $G$ is trivial, Lemma~\ref{lem.metric extension} reduces to the following (surprisingly unfamiliar) identity:
\[
\|au-bv\|^2
=(a-b)^2+ab\|u-v\|^2.
\]
In this setting, $a$ and $b$ need not be nonnegative for the identity to hold.

In what follows, we say a group $G$ acts \textbf{topologically transitively} on a topological space $X$ if there exists $x\in X$ such that the orbit $G\cdot x$ is dense in $X$.
For example, the group of rotations by rational multiples of $\pi$ acts topologically transitively on the unit circle.
By Lemma~\ref{lem.metric quotient orbit closure}, when $G$ acts by isometries on a metric space $X$, topological transitivity is equivalent to $X\dslash G$ being a singleton set.
Since bilipschitzness is ill-defined for maps on the singleton metric space, the following result isolates the topologically transitive case for completeness.

\begin{theorem}
\label{thm.bilipschitz homogeneous extension}
Consider any $G\leq O(V)$.
\begin{itemize}
\item[(a)]
If $G$ acts topologically transitively on $S(V)$, then for every $f\colon S(V)\dslash G\to S(W)$, the homogeneous extension $f^\star\colon V\dslash G\to W$ is an isometry.
\item[(b)]
If $G$ does not act topologically transitively on $S(V)$, then for every $f\colon S(V)\dslash G\to S(W)$ with optimal bilipschitz bounds $\alpha$ and $\beta$, the homogeneous extension $f^\star\colon V\dslash G\to W$ has optimal bilipschitz bounds $\min\{\alpha,1\}$ and $\max\{\beta,1\}$.
\end{itemize}
\end{theorem}

\begin{proof}
For (a), $S(V)\dslash G$ is a singleton set, so the image of $f$ consists of a single unit vector $w\in S(W)$.
Fix $u,v\in S(V)$ and $a,b\geq0$.
Then $d([u],[v])=0$, and so Lemma~\ref{lem.metric extension} gives 
\[
\|f^\star([au])-f^\star([bv])\|^2
=\|af([u])-bf([v])\|^2
=\|aw-bw\|^2
=(a-b)^2
=d([au],[bv])^2.
\]
For (b), fix $u,v\in S(V)$ and $a,b\geq0$.
Then Lemma~\ref{lem.metric extension} gives
\[
\|f^\star([au])-f^\star([bv])\|^2
=\|af([u])-bf([v])\|^2
=(a-b)^2+ab\|f([u])-f([v])\|^2.
\]
Since $f$ has lower Lipschitz bound $\alpha$, we then have
\begin{align*}
\|f^\star([au])-f^\star([bv])\|^2
&\geq (a-b)^2+ab \cdot \alpha^2 d([u],[v])^2\\
&\geq \min\{\alpha,1\}^2\cdot\Big((a-b)^2+ab \cdot d([u],[v])^2\Big)\\
&=\min\{\alpha,1\}^2\cdot d([au],[bv])^2,
\end{align*}
where the last step applies Lemma~\ref{lem.metric extension}.
In the case where $\min\{\alpha,1\}=1$, this bound is sharp when $[u]=[v]$ and $a\neq b$, by Lemma~\ref{lem.metric extension}.
In the case where $\min\{\alpha,1\}=\alpha$, this bound is sharp when $a=b=1$ since $\alpha$ is the optimal lower Lipschitz bound for $f$ by assumption.
A similar argument delivers the optimal upper Lipschitz bound.
\end{proof}

In practice, one might encounter $f$ with optimal lower Lipschitz bound $\alpha>1$, in which case the homogeneous extension $f^\star$ has distortion $\beta$, which is strictly larger than the distortion $\beta/\alpha$ of $f$.
In this case, we can first modify $f$ by lifting outputs into an extra dimension:

\begin{lemma}
\label{lem.lift}
{Let $X$ be a metric space.}
Given $f\colon X\to S(W)$ with optimal bilipschitz bounds $\alpha$ and $\beta$, then for each $t\in(0,1]$, the map $f_t\colon X\to S(W\oplus\mathbb{R})$ defined by $f_t(x)=(tf(x),\sqrt{1-t^2})$ has optimal bilipschitz bounds $t\alpha$ and $t\beta$.
\end{lemma}

\begin{proof}
The result follows from the fact that $\|f_t(x)-f_t(y)\|^2=t^2\|f(x)-f(y)\|^2$.
\end{proof}

Given $f\colon S(V)\dslash G\to S(W)$ with optimal bilipschitz bounds $\alpha>1$ and $\beta$, then for every $t\in[\beta^{-1},\alpha^{-1}]$, the homogeneous extension $(f_t)^\star$ also has distortion $\beta/\alpha$.
In the case where $\beta<1$, we do not have an approach to make homogeneous extension preserve distortion.
Next, we show how a generalization of the same lift can be used to convert a bounded bilipschitz map $S(V)\dslash G\to W$ into a bilipschitz map $S(V)\dslash G\to S(W\oplus\mathbb{R})$:

\begin{lemma}
\label{lem.lift to normalize}
{Let $X$ be a metric space.}
Given bounded $g\colon X\to W$ with bilipschitz bounds $\alpha$ and $\beta$, then for each $t\in(0,\|g\|_\infty^{-1})$, the map $g_t\colon X\to S(W\oplus\mathbb{R})$ defined by
$
g_t(x)
=(tg(x),\sqrt{1-\|tg(x)\|^2})
$
has (possibly suboptimal) bilipschitz bounds $t\alpha$ and $(1-t^2\|g\|_\infty^2)^{-1/2}\cdot t\beta$.
\end{lemma}

\begin{proof}
The lower Lipschitz bound follows from the fact that $\|g_t(x)-g_t(y)\|\geq t\|g(x)-g(y)\|$.
For the upper Lipschitz bound, put $r:=\|g\|_\infty$ and consider the map $h_t\colon[0,tr]\to\mathbb{R}$ defined by $h_t(u)=\sqrt{1-u^2}$.
The derivative of $h_t$ reveals that its optimal upper Lipschitz bound is $c_t:=(1-t^2r^2)^{-1/2}\cdot tr$.
We apply this bound and the reverse triangle inequality to obtain
\begin{align*}
\|g_t(x)-g_t(y)\|^2
&=t^2\|g(x)-g(y)\|^2+\big|h_t(\|tg(x)\|)-h_t(\|tg(y)\|)\big|^2\\
&\leq t^2\|g(x)-g(y)\|^2+c_t^2\cdot\big|\|tg(x)\|-\|tg(y)\|\big|^2
\leq t^2(1+c_t^2)\|g(x)-g(y)\|^2,
\end{align*}
and the result follows.
\end{proof}

Notably, Lemma~\ref{lem.lift to normalize} produces a map with near identical distortion when $t$ is small, but taking $t$ too small will send the upper Lipschitz bound below $1$, in which case homogeneous extension will increase the distortion.

In what follows, we describe a few examples of bilipschitz maps over $S(V)\dslash G$ that can be homogeneously extended to bilipschitz maps over $V\dslash G$ by Theorem~\ref{thm.bilipschitz homogeneous extension}(b).

\begin{example}[cf.\ Example~\ref{ex.finite dimensional real phase retrieval}]
\label{ex.real phase retrieval}
Fix a nontrivial real Hilbert space $V$, take $G:=\{ \pm\operatorname{id}\}\leq O(V)$, and consider $f\colon S(V)/G\to S(V^{\otimes 2})$ defined by $f([x])=x\otimes x$.
If $V$ has dimension~$1$, then $G$ acts transitively on $S(V)$, and so Theorem~\ref{thm.bilipschitz homogeneous extension}(a) applies.
Otherwise, for each $x,y\in S(V)$ with $[x]\neq[y]$, we have
\[
d([x],[y])^2
=\min\{\|x-y\|^2,\|x+y\|^2\}
=2-2|\langle x,y\rangle|,
\]
\[
\|f([x])-f([y])\|^2
=2-2\langle x\otimes x,y\otimes y\rangle
=2-2\langle x,y\rangle^2,
\]
and so taking $t:=|\langle x,y\rangle|$ gives
\[
\frac{\|f([x])-f([y])\|}{d([x],[y])}
=\sqrt{\frac{1-t^2}{1-t}}
=\sqrt{1+t}.
\]
Since $t$ takes values in $[0,1)$, it follows that $f$ has optimal bilipschitz bounds $1$ and $\sqrt{2}$. 
By Theorem~\ref{thm.bilipschitz homogeneous extension}(b), the homogeneous extension
\[
f^\star(x)
=\left\{\begin{array}{cl}\displaystyle\frac{x\otimes x}{\|x\|}&\text{if $x\neq0$}\vspace{0.1in}\\0&\text{otherwise}\end{array}\right.
\]
has the same optimal bilipschitz bounds.
By Corollary~\ref{cor.c2 real phase retrieval}, this is the minimum possible distortion for a Euclidean embedding of $V/G$.
By contrast, in the case where $V=\ell^2$, max filtering approaches fail to deliver bilipschitz invariants~\cite{CahilCD:16,AlaifariG:17}.
\end{example}

\begin{example}
Fix a nontrivial complex Hilbert space $V$, take $G:=\{\omega\cdot\operatorname{id}:|\omega|=1\}\leq U(V)$, and consider $f\colon S(V)/G\to S(V\otimes V^*)$ defined by $f([x])=x\otimes \langle x,\cdot\rangle$.
If $V$ has dimension~$1$, then $G$ acts transitively on $S(V)$, and so Theorem~\ref{thm.bilipschitz homogeneous extension}(a) applies.
Otherwise, following the argument in Example~\ref{ex.real phase retrieval}, then for any $x,y\in S(V)$, we have
\[
d([x],[y])^2
=2-2|\langle x,y\rangle|,
\qquad
\|f([x])-f([y])\|^2
=2-2|\langle x,y\rangle|^2,
\]
and $f$ has optimal bilipschitz bounds $1$ and $\sqrt{2}$. 
By Theorem~\ref{thm.bilipschitz homogeneous extension}(b), the homogeneous extension $f^\star$ has the same optimal bilipschitz bounds.
By Corollary~\ref{cor.c2 complex phase retrieval}, this is the minimum distortion for a Euclidean embedding of $V/G$.
\end{example}

\begin{example}
\label{ex.discrete rotations}
Fix $V:=\mathbb{C}$, take $G:=\langle e^{2\pi i/r}\rangle\leq U(1)$ for some $r\in\mathbb{N}$, and consider $f\colon S(V)/G\to S(V)$ defined by $f([x])=x^r$.
Then for every $x,y\in S(V)$, we have
\[
d([x],[y])^2
=2-2\max_{g\in G}\operatorname{Re}(gx\overline{y}),
\qquad
\|f([x])-f([y])\|^2
=2-2\operatorname{Re}((x\overline{y})^r).
\]
When $[x]\neq[y]$, we take $h\in G$ such that $hx\overline{y}=e^{2\pi it}$ for some $t\in[-\frac{1}{2r},\frac{1}{2r})\setminus\{0\}$.
Then
\[
\frac{\|f([x])-f([y])\|}{d([x],[y])}
=\sqrt{\frac{1-\operatorname{Re}((x\overline{y})^r)}{1-\displaystyle\max_{g\in G}\operatorname{Re}(gx\overline{y})}}
=\sqrt{\frac{1-\operatorname{Re}((hx\overline{y})^r)}{1-\displaystyle\max_{g\in G}\operatorname{Re}(ghx\overline{y})}}
=\frac{\sin(\pi rt)}{\sin(\pi t)}
=:g(t).
\]
We note that $g$ is decreasing on $(0,\frac{1}{2r})$ since $g'(t)<0$ precisely when $\frac{\tan(\pi t)}{\pi t}<\frac{\tan(\pi rt)}{\pi rt}$, while $\frac{d}{dx}(\frac{\tan(x)}{x})>0$ for $x\in(0,\frac{\pi}{2})$ since $\sin(x)\cos(x)<x$.
Since $g$ is even, evaluating at $t=\frac{1}{2r}$ and taking the limit $t\to0$ gives the optimal bilipschitz bounds, namely, $\csc(\frac{\pi}{2r})$ and $r$.
Notably, $\csc(\frac{\pi}{2r})>1$ when $r>1$, so one should lift as in Lemma~\ref{lem.lift} before taking the homogeneous extension in order to preserve the distortion $r\sin(\frac{\pi}{2r})\in[1,\frac{\pi}{2})$.
By Corollary~\ref{cor.c2 discrete rotations}, this is the minimum distortion for a Euclidean embedding of $V/G$.
\end{example}

\section{Non-differentiability of bilipschitz invariants}
\label{sec.nondiff bilipschitz}

We motivate this section with an example:

\begin{example}
\label{ex.max filtering is not differentiable}
By Proposition~\ref{prop.max filtering bilipschitz}, max filtering produces bilipschitz invariants modulo finite groups of $O(d)$.
However, these invariants fail to be differentiable.
Indeed, $x\mapsto\llangle [z],[x]\rrangle$ is non-differentiable precisely at the boundaries of the Voronoi diagram of the orbit $G\cdot z$.
For an explicit example, take $V:=\mathbb{R}^2$, templates $z_1:=(1,1)$, $z_2:=(1,2)$, and $z_3:=(3,1)$, and given $G\leq O(2)$, consider the max filtering invariant
\[
f\colon
x
\mapsto
\Big(~\llangle [z_1],[x]\rrangle,~\llangle [z_2],[x]\rrangle,~\llangle [z_3],[x]\rrangle~\Big).
\]
We consider two cases: when $G$ consists of rotations by multiples of $2\pi/3$, and when $G$ is the dihedral group of order $6$.
See Figure~\ref{fig.points of non-differentiability} for an illustration.
The points at which $f$ is not differentiable form rays emanating from the origin.
Later in this section, we show that every point with a nontrivial stabilizer in $G$ (i.e., every \textit{non-principal point}) is necessarily a point of non-differentiability in any bilipschitz invariant; see Theorem~\ref{thm.bilipschitz requires nondifferentiability}(b).
In both parts of Figure~\ref{fig.points of non-differentiability} (and in fact, whenever $G$ is nontrivial), the origin is one such non-principal point.
In the example on the left, this is the only non-principal point, whereas on the right, every non-differentiable point is a non-principal point.
In what follows, we show how bilipschitzness \textit{requires} such non-differentiability of max filtering invariants.
\end{example}

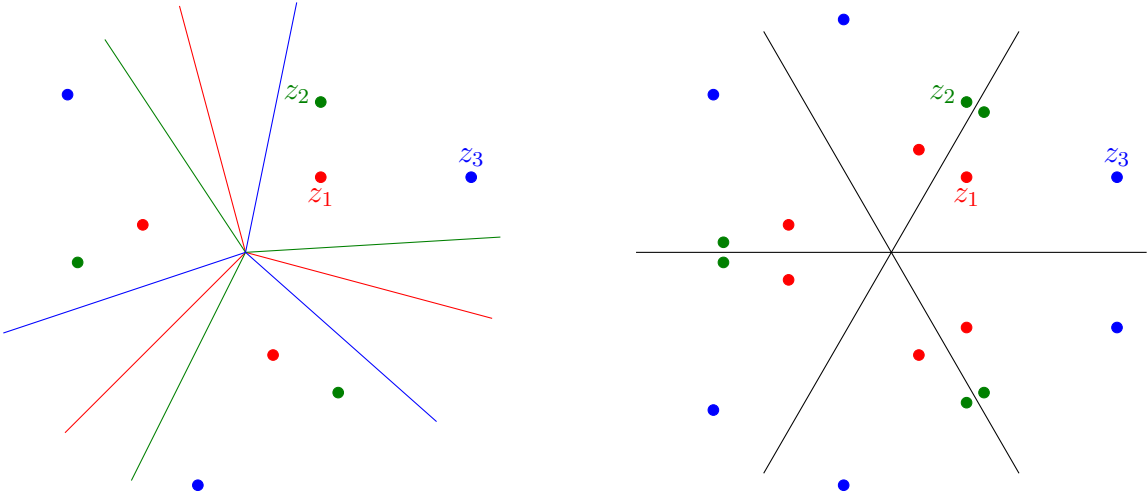
\begin{figure}
\begin{center}
\begin{tikzpicture}
\coordinate (O) at (0,0);
\coordinate (z1) at (1,1);
\coordinate (z2) at (1,2);
\coordinate (z3) at (3,1);
\coordinate (Rz1) at ({-1/2-0.5*sqrt(3)},{-1/2+0.5*sqrt(3)});
\coordinate (R2z1) at ({-1/2+0.5*sqrt(3)},{-1/2-0.5*sqrt(3)});
\coordinate (Rz2) at ({-1/2-sqrt(3)},{-1+0.5*sqrt(3)});
\coordinate (R2z2) at ({-1/2+sqrt(3)},{-1-0.5*sqrt(3)});
\coordinate (Rz3) at ({-1.5-0.5*sqrt(3)},{-1/2+1.5*sqrt(3)});
\coordinate (R2z3) at ({-1.5+0.5*sqrt(3)},{-1/2-1.5*sqrt(3)});
\draw[color=red] (O) -- ({1.2*(1-sqrt(3))},{1.2*(1+sqrt(3))});
\draw[color=red] (O) -- (-1.2*2,-1.2*2);
\draw[color=red] (O) -- ({1.2*(1+sqrt(3))},{1.2*(1-sqrt(3))});
\draw[color=darkgreen] (O) -- ({1.2*sqrt(8/5)*(0.5-sqrt(3))},{1.2*sqrt(8/5)*(1+0.5*sqrt(3))});
\draw[color=darkgreen] (O) -- ({1.2*sqrt(8/5)*(-1)},{1.2*sqrt(8/5)*(-2)});
\draw[color=darkgreen] (O) -- ({1.2*sqrt(8/5)*(0.5+sqrt(3))},{1.2*sqrt(8/5)*(1-0.5*sqrt(3))});
\draw[color=blue] (O) -- ({1.2*sqrt(8/10)*(1.5-0.5*sqrt(3))},{1.2*sqrt(8/10)*(0.5+1.5*sqrt(3))});
\draw[color=blue] (O) -- ({1.2*sqrt(8/10)*(-3)},{1.2*sqrt(8/10)*(-1)});
\draw[color=blue] (O) -- ({1.2*sqrt(8/10)*(1.5+0.5*sqrt(3))},{1.2*sqrt(8/10)*(0.5-1.5*sqrt(3))});
\draw[fill,color=red] (z1) circle (2 pt) node [below] {$z_1$};
\draw[fill,color=darkgreen] (z2) circle (2 pt) node [xshift=-9,yshift=3] {$z_2$};
\draw[fill,color=blue] (z3) circle (2 pt) node [above] {$z_3$};
\draw[fill,color=red] (Rz1) circle (2 pt);
\draw[fill,color=red] (R2z1) circle (2 pt);
\draw[fill,color=darkgreen] (Rz2) circle (2 pt);
\draw[fill,color=darkgreen] (R2z2) circle (2 pt);
\draw[fill,color=blue] (Rz3) circle (2 pt);
\draw[fill,color=blue] (R2z3) circle (2 pt);
\end{tikzpicture}
\qquad\qquad
\begin{tikzpicture}
\coordinate (O) at (0,0);
\coordinate (z1) at (1,1);
\coordinate (z2) at (1,2);
\coordinate (z3) at (3,1);
\coordinate (Rz1) at ({-1/2-0.5*sqrt(3)},{-1/2+0.5*sqrt(3)});
\coordinate (R2z1) at ({-1/2+0.5*sqrt(3)},{-1/2-0.5*sqrt(3)});
\coordinate (Rz2) at ({-1/2-sqrt(3)},{-1+0.5*sqrt(3)});
\coordinate (R2z2) at ({-1/2+sqrt(3)},{-1-0.5*sqrt(3)});
\coordinate (Rz3) at ({-1.5-0.5*sqrt(3)},{-1/2+1.5*sqrt(3)});
\coordinate (R2z3) at ({-1.5+0.5*sqrt(3)},{-1/2-1.5*sqrt(3)});
\coordinate (Fz1) at (1,-1);
\coordinate (Fz2) at (1,-2);
\coordinate (Fz3) at (3,-1);
\coordinate (FRz1) at ({-1/2-0.5*sqrt(3)},{1/2-0.5*sqrt(3)});
\coordinate (FR2z1) at ({-1/2+0.5*sqrt(3)},{1/2+0.5*sqrt(3)});
\coordinate (FRz2) at ({-1/2-sqrt(3)},{1-0.5*sqrt(3)});
\coordinate (FR2z2) at ({-1/2+sqrt(3)},{1+0.5*sqrt(3)});
\coordinate (FRz3) at ({-1.5-0.5*sqrt(3)},{1/2-1.5*sqrt(3)});
\coordinate (FR2z3) at ({-1.5+0.5*sqrt(3)},{1/2+1.5*sqrt(3)});
\draw[] ({-1.2*sqrt(8)},0) -- ({1.2*sqrt(8)},0);
\draw[] ({1.2*sqrt(8)*0.5},{1.2*sqrt(8)*0.5*sqrt(3)}) -- ({-1.2*sqrt(8)*0.5},{-1.2*sqrt(8)*0.5*sqrt(3)});
\draw[] ({-1.2*sqrt(8)*0.5},{1.2*sqrt(8)*0.5*sqrt(3)}) -- ({1.2*sqrt(8)*0.5},{-1.2*sqrt(8)*0.5*sqrt(3)});
\draw[fill,color=red] (z1) circle (2 pt) node [below] {$z_1$};
\draw[fill,color=darkgreen] (z2) circle (2 pt) node [xshift=-9,yshift=3] {$z_2$};
\draw[fill,color=blue] (z3) circle (2 pt) node [above] {$z_3$};
\draw[fill,color=red] (Rz1) circle (2 pt);
\draw[fill,color=red] (R2z1) circle (2 pt);
\draw[fill,color=darkgreen] (Rz2) circle (2 pt);
\draw[fill,color=darkgreen] (R2z2) circle (2 pt);
\draw[fill,color=blue] (Rz3) circle (2 pt);
\draw[fill,color=blue] (R2z3) circle (2 pt);
\draw[fill,color=red] (Fz1) circle (2 pt);
\draw[fill,color=darkgreen] (Fz2) circle (2 pt);
\draw[fill,color=blue] (Fz3) circle (2 pt);
\draw[fill,color=red] (FRz1) circle (2 pt);
\draw[fill,color=red] (FR2z1) circle (2 pt);
\draw[fill,color=darkgreen] (FRz2) circle (2 pt);
\draw[fill,color=darkgreen] (FR2z2) circle (2 pt);
\draw[fill,color=blue] (FRz3) circle (2 pt);
\draw[fill,color=blue] (FR2z3) circle (2 pt);
\end{tikzpicture}
\end{center}
\caption{\label{fig.points of non-differentiability}
Illustration of Example~\ref{ex.max filtering is not differentiable}.
Max filtering invariants are piecewise linear.
In this example, the orbits of the templates $z_1,z_2,z_3\in\mathbb{R}^2$ determine points of non-differentiability (namely, the boundaries of their Voronoi cells) in the resulting max filtering invariant.
On the left, the group $G\leq O(2)$ consists of rotations by multiples of $2\pi/3$, while on the right, $G$ is the dihedral group of order $6$.
Considering Theorem~\ref{thm.bilipschitz requires nondifferentiability}(b), the non-differentiability of max filtering is an artifact of its bilipschitzness.}
\end{figure}

Recall that a function $f\colon V\to W$ is \textbf{Fr\'{e}chet differentiable} at $x\in V$ if there exists a bounded linear map $A\colon V\to W$ such that 
\[
\lim_{h\to0}\frac{\|f(x+h)-f(x)-Ah\|}{\|h\|}
=0,
\]
in which case we say $A$ is the \textbf{Fr\'{e}chet derivative} at $x$, which we denote by $Df(x):=A$; see~\cite{CartanMHM:17} for more information.
If $V$ and $W$ are Euclidean, then the matrix representation of $Df(x)$ is the Jacobian of $f$ at $x$.

\begin{lemma}
\label{lem.derivative of invariant map}
Suppose $f\colon V\to W$ is $G$-invariant and Fr\'{e}chet differentiable at $x\in V$.
Then for each $g\in G$, it holds that $f$ is Fr\'{e}chet differentiable at $gx$ with $Df(gx)=Df(x)\circ g^{-1}$.
\end{lemma}

\begin{proof}
By $G$-invariance, we have $f(gx)=f(x)$ and $f(gx+h)=f(x+g^{-1}h)$, and so
\[
\lim_{h\to0}\frac{\|f(gx+h)-f(gx)-(Df(x)\circ g^{-1})h\|}{\|h\|}
=\lim_{h\to0}\frac{\|f(x+g^{-1}h)-f(x)-Df(x)(g^{-1}h)\|}{\|g^{-1}h\|}.
\]
Change variables $g^{-1}h\mapsto h$ and apply Fr\'{e}chet differentiability at $x$ to obtain the result. 
\end{proof}

\begin{theorem}
\label{thm.bilipschitz requires nondifferentiability}
Suppose $x\in V$ is fixed by some nonidentity member of $G$, and consider any $G$-invariant map $f\colon V\to W$ that is Fr\'{e}chet differentiable at $x$.
Then the following hold.
\begin{itemize}
\item[(a)]
There exists a unit vector $v\in V$ orthogonal to $x$ such that $Df(x)v=0$.
\item[(b)]
If $G$ is finite or $x=0$, then $f^\downarrow\colon V\dslash G\to W$ is not lower Lipschitz.
\item[(c)]
If $G$ is finite and $x\in S(V)$, then the restriction $f^\downarrow|_{S(V)/G}$ is not lower Lipschitz.
\end{itemize}
\end{theorem}

In words, a bilipschitz invariant cannot be differentiable at any point with a nontrivial stabilizer.
In particular, if $G$ is nontrivial, no bilipschitz invariant is differentiable at $0$.
For each example in Section~\ref{sec:homogeneous extension}, $G$ acts freely on $V\setminus\{0\}$, and so we can get away with $0$ being the only point at which $f$ is not differentiable.

\begin{proof}[Proof of Theorem~\ref{thm.bilipschitz requires nondifferentiability}]
For (a), fix a nontrivial $g\in G$ with $gx=x$.
Since $\operatorname{ker}(g-\operatorname{id})$ is a closed and proper subspace of $V$, there is a unit vector $v$ in $\operatorname{ker}(g-\operatorname{id})^\perp$.
Notably, $v$ is orthogonal to $x\in\operatorname{ker}(g-\operatorname{id})$.
It remains to show $Df(x)v=0$.
By Lemma~\ref{lem.derivative of invariant map}, $Df(x)=Df(gx)=Df(x)\circ g^{-1}$, and so $Df(x)^*=g\circ Df(x)^*$.
Thus, $g$ fixes every element of $\operatorname{im}Df(x)^*$ so that $\operatorname{im}Df(x)^*\leq\operatorname{ker}(g-\operatorname{id})$.
Then $v\in\operatorname{ker}(g-\operatorname{id})^\perp\leq(\operatorname{im}Df(x)^*)^\perp=\operatorname{ker}Df(x)$.

For (b), take $v\in V$ from (a).
Since $G$ is finite or $x=0$, either the members of $[x]$ have some minimum pairwise distance $\delta>0$, or $[x]$ is a singleton set and we may take $\delta:=\infty$ in what follows.
For every $t\in\mathbb{R}$ such that $|t|<\delta/2$, it holds that $d([x+tv],[x])=|t|$, since every $y\in[x]\setminus\{x\}$ satisfies $\|(x+tv)-y\|\geq||t|-\|y-x\||>|t|$.
The definition of the Fr\'{e}chet derivative then implies
\[
\liminf_{t\to0}\frac{\|f^\downarrow([x+tv])-f^\downarrow([x])\|}{d([x+tv],[x])}
=\liminf_{t\to0}\frac{\|f(x+tv)-f(x)-Df(x)tv\|}{|t|}
=0.
\]
Thus, $f^\downarrow$ is not lower Lipschitz.

For (c), we argue as in (b) with a different perturbation of $x$.
Take $v$ from (a) and put $h(t):=x\cos(t)+v\sin(t)-x$ so that $x+h(t)\in S(V)$ for all $t$.
An argument like above shows $d([x+h(t)],[x])=\|h(t)\|$ when $|t|$ is sufficiently small.
We apply the triangle inequality, the definition of the Fr\'{e}chet derivative, and the fact that $Df(x)v=0$ to get
\begin{align*}
&\liminf_{t\to0}\frac{\|f^\downarrow([x+h(t)])-f^\downarrow([x])\|}{d([x+h(t)],[x])}
=\liminf_{t\to0}\frac{\|f(x+h(t))-f(x)\|}{\|h(t)\|}\\
&\leq\liminf_{t\to0}\frac{\|f(x+h(t))-f(x)-Df(x)h(t)\|+\|Df(x)(h(t)-tv)\|+\|Df(x)tv\|}{\|h(t)\|}\\
&=\liminf_{t\to0}\frac{\|Df(x)(h(t)-tv)\|}{\|h(t)\|}
\leq\|Df(x)\|\cdot\liminf_{t\to0}\frac{\|h(t)-tv\|}{\|h(t)\|}.
\end{align*}
Considering $\|h(t)-tv\|^2=(\cos t-1)^2+(\sin t-t)^2$ and $\|h(t)\|^2=(\cos t-1)^2+\sin^2 t$, the $\liminf$ above is zero, and so $f^\downarrow|_{S(V)/G}$ is not lower Lipschitz.
\end{proof}

\begin{example}
As an application, consider the problem of multi-reference alignment~\cite{BendoryBMZS:17}.
Here, $G$ is the group of cyclic permutations, and the task is to estimate an unknown $[x]\in\mathbb{C}^d/G$ given data of the form $\{g_ix+z_i\}_{i=1}^n$, where the $g_i$'s are drawn uniformly from $G$ and the $z_i$'s are independent complex gaussian vectors with large variance.
If $G$ were trivial, then one could estimate $x$ with the sample average.
Since $G$ is nontrivial, we are inclined to ``average out the noise'' in an invariant domain.

One popular invariant is the \textit{bispectrum} $B\colon\mathbb{C}^d\to\mathbb{C}^{d^2}$ defined by
\[
B(x)_{k\ell}
=\hat{x}(k)\cdot\overline{\hat{x}(\ell)}\cdot\hat{x}(\ell-k),
\]
where $\hat{x}$ denotes the discrete Fourier transform of $x$, and $\ell-k$ is interpreted modulo $d$.
The approach in \cite{BendoryBMZS:17} performs multi-reference alignment by first using the data $\{g_ix+z_i\}_{i=1}^n$ to find an estimate $y$ of $B(x)$, and then using $y$ to estimate $[x]$.
For the second step, they construct a function $\psi\colon(\mathbb{C}^\times)^{d^2}\to\mathbb{C}^d/G$ such that $\psi(B^\downarrow([x]))=[x]$ whenever $\hat{x}$ is everywhere nonzero, and they show that $\psi$ is \textit{locally} Lipschitz.

In what follows, we show that $\psi$ is not Lipschitz as a consequence of Theorem~\ref{thm.bilipschitz requires nondifferentiability}(b).
Suppose otherwise that $\psi$ is $\beta$-Lipschitz, and let $U$ denote the vectors in $\mathbb{C}^d$ with everywhere nonzero discrete Fourier transform.
Since $U$ is dense in $\mathbb{C}^d$ and $B$ (and hence $B^\downarrow$) is continuous, we have
\[
\inf_{\substack{x,y\in\mathbb{C}^d\\{[x]\neq[y]}}}
\frac{\|B^\downarrow([x])-B^\downarrow([y])\|}{d([x],[y])}
=\inf_{\substack{x,y\in U\\{[x]\neq[y]}}}
\frac{\|B^\downarrow([x])-B^\downarrow([y])\|}{d([x],[y])}
=\inf_{\substack{x,y\in U\\{[x]\neq[y]}}}
\frac{\|B^\downarrow([x])-B^\downarrow([y])\|}{d(\psi(B^\downarrow([x])),\psi(B^\downarrow([y])))}
\geq \frac{1}{\beta},
\]
i.e., $B^\downarrow$ is $1/\beta$-lower Lipschitz.
Since $G$ fixes the constant functions and $B$ is differentiable as a map between real Hilbert spaces, this contradicts Theorem~\ref{thm.bilipschitz requires nondifferentiability}(b).
{(Note that in this argument, the points used to break a Lipschitz bound on $\psi$ are precisely those with vanishing Fourier coefficients, perhaps because at least some of these orbits are not even separated by $B$.)}
In general, a left inverse to a translation-invariant map $f$ cannot be Lipschitz unless $f$ is not differentiable at the constant functions (and at any other nontrivially periodic function).
Such non-differentiability is afforded by max filtering, which in turn delivers bilipschitz invariants~\cite{CahillIMP:22}.
\end{example}

\section{Bilipschitz polynomial invariants}
\label{sec.bilipschitz poly}

Classical invariant theory is concerned with polynomial maps, thanks in part to the following well-known result.

\begin{proposition}
\label{prop.generators separate}
For each finite $G\leq GL(d)$, there exists an injective polynomial map $\mathbb{R}^d/G\to\mathbb{R}^n$ for some $n\in\mathbb{N}$.
\end{proposition}

\begin{proof}
{First, the $G$-invariant polynomials separate all $G$-orbits by Theorem 3 in Ch.\ 3, Sec.\ 4 of~\cite{OnishchikV:90}.
Next, Hilbert~\cite{Hilbert:93} established that the algebra of $G$-invariant polynomials is generated by finitely many polynomials $\{h_i(x)\}_{i=1}^n$.
Defining $h\colon\mathbb{R}^d\to\mathbb{R}^n$ by $h(x):=(h_1(x),\ldots,h_n(x))$, then $h^\downarrow$ is the desired injection.}
%First, we observe that every pair of distinct $G$-orbits $[x],[y]\in\mathbb{R}^d/G$ is necessarily separated by some $G$-invariant polynomial function $f\colon\mathbb{R}^d\to\mathbb{R}$.
%For example, applying the Reynolds operator $h\mapsto \frac{1}{|G|}\sum_{g\in G}(g\cdot h)$ to the map $h\colon z\mapsto \prod_{g\in G}\|z-gx\|^2$ gives such an $f$, since then $f(z)\geq0$ with equality precisely when $z\in[x]$.
%Next, Hilbert~\cite{Hilbert:93} established that the algebra of $G$-invariant polynomials with complex coefficients is generated by finitely many polynomials $\{h_i(x)\}_{i=1}^k$.
%Thus, every pair of distinct $G$-orbits in $\mathbb{R}^d$ is necessarily separated by some $h_i\colon\mathbb{R}^d\to\mathbb{C}$.
%Defining $h\colon\mathbb{R}^d\to\mathbb{C}^k\cong\mathbb{R}^{2k}$ by $h(x):=(h_1(x),\ldots,h_k(x))$, then $h^\downarrow$ is the desired injection.
\end{proof}

In this section, we evaluate polynomial maps in terms of bilipschitzness.

\begin{lemma}
Suppose $G\leq O(d)$ is nontrivial and $f\colon \mathbb{R}^d\to\mathbb{R}^n$ is a $G$-invariant polynomial.
Then $f^\downarrow\colon\mathbb{R}^d\dslash G\to\mathbb{R}^n$ is not lower Lipschitz.
Furthermore, $f^\downarrow$ is upper Lipschitz only if $f$ is affine linear, in which case $f^\downarrow$ is not injective.
\end{lemma}

\begin{proof}
Since the origin has a nontrivial stabilizer and $f$ is differentiable at the origin, Theorem~\ref{thm.bilipschitz requires nondifferentiability} gives that $f^\downarrow$ is not lower Lipschitz.

Next, suppose $f$ is not affine linear.
We will show that $f^\downarrow$ is not upper Lipschitz.
Select a coordinate function $f_i(x)$ of degree $k\geq2$, and let $g(x)$ denote the homogeneous component of $f_i(x)$ of degree $k$.
Since $g(x)$ is nonzero, there exists $v\in\mathbb{R}^d$ such that $g(v)\neq0$.
Note that $g(0)=0$ by homogeneity, and so $v\neq0$.
Then $g(tv)=t^kg(v)$, and so $f_i(tv)=t^kg(v)+O(t^{k-1})$ for large $t$.
Thus, $\|f^\downarrow([tv])-f^\downarrow([0])\|\geq|f_i(tv)-f_i(0)|= |t^kg(v)+O(t^{k-1})|= t^k|g(v)|+O(t^{k-1})$ for large $t$.
Since $d([tv],[0])=|t|\|v\|$, it follows that $f^\downarrow$ is not upper Lipschitz.

Finally, suppose $f$ is affine linear and write $f(x)=Ax+b$.
Since $G$ is nontrivial, there exists $g\in G$ and $y\in\mathbb{R}^d$ such that $z:=gy\neq y$.
Then $f(y)=f(z)$, meaning $Ay+b=Az+b$, and so $A(y-z)=0$.
It follows that $f^\downarrow([y-z])=f^\downarrow([0])$, i.e., $f^\downarrow$ is not injective.
\end{proof}

While we cannot expect polynomial invariants to be bilipschitz, there is some hope of applying ideas from Section~\ref{sec:homogeneous extension} to obtain bilipschitz maps from polynomial invariants by homogeneous extension.
In fact, all of the examples from Section~\ref{sec:homogeneous extension} were obtained in this way.
Considering Lemma~\ref{lem.lift to normalize} {and Theorem~\ref{thm.bilipschitz homogeneous extension}(b)}, it suffices to seek $G$-invariant polynomial maps $f\colon \mathbb{R}^d\to\mathbb{R}^n$ for which $f^\downarrow|_{S(\mathbb{R}^d)\dslash G}$ is bilipschitz.
By Theorem~\ref{thm.bilipschitz requires nondifferentiability}(c), this is not possible when $G$ is finite unless it acts freely on $S(\mathbb{R}^d)$.
(Indeed, for each example in Section~\ref{sec:homogeneous extension}, $G$ acts freely.)
This is a strong condition, and it implies that every abelian subgroup of $G$ is cyclic; see Theorems~5.3.1 and~5.3.2 in~\cite{Wolf:74}.
Nevertheless, nonabelian examples exist, and furthermore, they have been classified; see Chapter~6 of~\cite{Wolf:74}.
Interestingly, every one of these examples admits a polynomial $f\colon \mathbb{R}^d\to\mathbb{R}^n$ for which $f^\downarrow|_{S(\mathbb{R}^d)/G}$ is bilipschitz:

\begin{theorem}
\label{thm.poly free}
For finite $G\leq O(d)$, the following are equivalent:
\begin{itemize}
\item[(a)]
$G$ acts freely on $S(\mathbb{R}^d)$.
\item[(b)]
There exists a bilipschitz polynomial map $S(\mathbb{R}^d)/G\to\mathbb{R}^n$ for some $n\in\mathbb{N}$.
\item[(c)]
There exists a bilipschitz polynomial map $S(\mathbb{R}^d)/G\to\mathbb{R}^{2d-1}$.
\end{itemize}
\end{theorem}

If $G\leq O(d)$ is finite and acts freely on $S(\mathbb{R}^d)$, then one may combine Theorem~\ref{thm.poly free}(c) with Lemma~\ref{lem.lift to normalize} and Theorem~\ref{thm.bilipschitz homogeneous extension}(b) to obtain a bilipschitz map $\mathbb{R}^d/G\to\mathbb{R}^{2d}$.
By {comparison}, max filtering delivers injective maps $\mathbb{R}^d/G\to\mathbb{R}^{2d}$, and while these maps are conjectured to be bilipschitz, this is currently only known for certain choices of $G$~\cite{CahillIMP:22,MixonQ:22}.
(Max filtering is known to provide bilipschitz invariants $\mathbb{R}^d/G\to\mathbb{R}^n$ for large $n$.)
We will prove Theorem~\ref{thm.poly free} with the help of a few (essentially known) propositions, {whose proofs can be found in the appendix.}

\begin{proposition}
\label{prop.lagrange inperpolation}
Given distinct $u_1,\ldots,u_n\in\mathbb{R}^d$ and not necessarily distinct $v_1,\ldots,v_n\in\mathbb{R}^d$, there exists a polynomial function $p\colon\mathbb{R}^d\to\mathbb{R}$ such that $\nabla p(u_i)=v_i$ for every $i\in\{1,\ldots,n\}$.
\end{proposition}

\begin{proposition}
\label{prop.quotient immersion}
Let $G$ be a finite group acting smoothly and freely on a smooth manifold $M$, and consider the quotient map $\pi\colon M\to M/G$ defined by $\pi(x)=[x]$.
\begin{itemize}
\item[(a)]
$M/G$ is a topological manifold of the same dimension as $M$.
\item[(b)]
There exists a unique smooth structure on $M/G$ for which $\pi$ is a smooth submersion.
\item[(c)]
For any $G$-invariant smooth immersion $g\colon M\to\mathbb{R}^n$, the corresponding map $g^\downarrow\colon M/G\to\mathbb{R}^n$ is a smooth immersion.
\end{itemize}
If in addition $M$ is a connected Riemannian manifold and $G$ acts by isometries\footnote{{In Riemannian geometry, an isometry is a smooth map between Riemannian manifolds that preserves the Riemannian metric tensor, whereas in this paper, an isometry is any map between metric spaces that preserves the distance.
Thankfully, these notions coincide for maps between Riemannian manifolds by the Myers--Steenrod Theorem.}}, then the following also hold:
\begin{itemize}
\item[(d)]
$M/G$ has a unique Riemannian metric tensor such that $\pi$ is a Riemannian covering.
\item[(e)]
The Riemannian distance on $M/G$ coincides with the quotient metric $d_{M/G}$.
\end{itemize}
\end{proposition}

\begin{proposition}
\label{prop.embedding is bilipschitz}
For Riemannian manifolds $X$ and $Y$ with $X$ compact, every smooth embedding $X\to Y$ is bilipschitz.
\end{proposition}

\begin{proposition}
\label{prop.generic projection of semialgebraic manifold}
Given a $k$-dimensional {compact smooth} semialgebraic submanifold $M$ of $\mathbb{R}^n$ and $m>2k$, then for a generic linear map $L\colon\mathbb{R}^n\to\mathbb{R}^m$, the restriction $L|_M$ is a {smooth} embedding.
\end{proposition}

\begin{proof}[Proof of Theorem~\ref{thm.poly free}]
{First, by Theorem~\ref{thm.bilipschitz requires nondifferentiability}(c), we have both (b)$\Rightarrow$(a) and (c)$\Rightarrow$(a).}

For (a)$\Rightarrow$(b), we first claim that for every $u\in S(\mathbb{R}^d)$, there exists a $G$-invariant polynomial map $p_u\colon\mathbb{R}^d\to\mathbb{R}^d$ such that $Dp_u(u)$ is invertible.
To see this, for each $i\in\{1,\ldots,d\}$, take any polynomial $q_{u,i}\colon\mathbb{R}^d\to\mathbb{R}$ such that $\nabla q_{u,i}(gu)=ge_i$ for all $g\in G$.
Such a polynomial exists by Proposition~\ref{prop.lagrange inperpolation}; indeed, $\{gu\}_{g\in G}$ are distinct since $G$ acts freely on $S(\mathbb{R}^d)$ by assumption.
Let $\overline{q}_{u,i}$ denote the $G$-invariant polynomial map obtained by applying the Reynolds operator to $q_{u,i}$.
Then
\[
\nabla\overline{q}_{u,i}(x)^\top
=\frac{1}{|G|}\sum_{g\in G}\nabla(q_{u,i}\circ g)(x)^\top
=\frac{1}{|G|}\sum_{g\in G}\nabla q_{u,i}(gx)^\top g.
\]
Evaluating at $x=u$ then gives
\[
\nabla\overline{q}_{u,i}(u)^\top
=\frac{1}{|G|}\sum_{g\in G}\nabla q_{u,i}(gu)^\top g
=\frac{1}{|G|}\sum_{g\in G} (ge_i)^\top g
=e_i^\top.
\]
Then taking $p_u(x):=(\overline{q}_{u,1}(x),\ldots,\overline{q}_{u,d}(x))$ gives $Dp_u(u)=\operatorname{id}$, as desired.

Next, by the continuity of $x\mapsto\operatorname{det}(Dp_u(x))$, we have that $Dp_u(x)$ is invertible on an open neighborhood $N_u$ of $u$.
By compactness, the open cover $\{N_u\}_{u\in S(\mathbb{R}^d)}$ of $S(\mathbb{R}^d)$ has a finite subcover $\{N_u\}_{u\in F}$.
Then the polynomial map $r_1\colon x\mapsto\{p_u(x)\}_{u\in F}$ has the property that $Dr_1(x)$ is injective for every $x\in S(\mathbb{R}^d)$.
Next, let $r_2$ denote any polynomial $G$-invariant map that separates $G$-orbits as in Proposition~\ref{prop.generators separate}.
The polynomial map $r\colon S(\mathbb{R}^d)\to\mathbb{R}^n$ defined by $r(x)=(r_1(x),r_2(x))$ is an immersion due to the $r_1$ component, and so Proposition~\ref{prop.quotient immersion} gives that {$S(\mathbb{R}^d)/G$ is a smooth manifold and} $r^\downarrow$ is a {smooth} immersion.
Furthermore, the $r_2$ component ensures that $r^\downarrow$ is injective.
As such, $r^\downarrow$ is a {smooth} embedding {by Proposition~4.22 in~\cite{Lee:13}}.
Finally, the fact that $r^\downarrow$ is bilipschitz follows from Proposition~\ref{prop.embedding is bilipschitz}.
Specifically, to use this result, we note {two things.
First, since $G$ acts freely, $S(\mathbb{R}^d)/G$ is a Riemannian manifold by Proposition~\ref{prop.quotient immersion}.
Second,} the Euclidean {distance} on $S(\mathbb{R}^d)$ is equivalent to the standard Riemannian {distance}, and so the quotient Euclidean {distance} on $S(\mathbb{R}^d)/G$ is equivalent to the quotient Riemannian {distance}, which in turn {equals the} Riemannian {distance} on $S(\mathbb{R}^d)/G$ {by Proposition~\ref{prop.quotient immersion}}.

{Finally, for (a)$\Rightarrow$(c), since $G$ acts freely, we have that $S(\mathbb{R}^d)/G$ is a smooth manifold by Proposition~\ref{prop.quotient immersion}.
Next, the above shows that} we have a bilipschitz polynomial map $r^\downarrow\colon S(\mathbb{R}^d)/G\to \mathbb{R}^n$.
The lower Lipschitz bound ensures that $r^\downarrow$ is a {smooth} embedding.
As such, $M:=\operatorname{im}(r^\downarrow)$ is a compact {smooth} submanifold of $\mathbb{R}^n$ of dimension $d-1$.
Furthermore, $M$ is the image of the semialgebraic set $S(\mathbb{R}^d)$ under a polynomial map, and so it is semialgebraic.
By Proposition~\ref{prop.generic projection of semialgebraic manifold}, $L|_M$ is a {smooth} embedding for a generic linear map $L\colon\mathbb{R}^n\to\mathbb{R}^{2d-1}$, in which case $L\circ r^\downarrow$ is the desired map.
(Indeed, bilipschitzness follows from Proposition~\ref{prop.embedding is bilipschitz} as before.)
\end{proof}

While our proof of (a)$\Rightarrow$(b) in Theorem~\ref{thm.poly free} was not constructive, the following result provides a construction in the special case where $G$ is abelian.
Considering $G\leq O(d)\leq U(d)$, we may focus on the complex case without loss of generality.
Then the spectral theorem affords $\mathbb{C}^d$ with an orthonormal basis that simultaneously diagonalizes every element of $G$, as in the hypotheses below.

\begin{theorem}
\label{thm.constructive invariants in abelian case}
Given a finite abelian $G\leq U(d)$, choose any orthonormal basis $\{u_i\}_{i=1}^d$ of $\mathbb{C}^d$ for which there exists characters $\{\chi_i\}_{i=1}^d$ of $G$ such that $gu_i=\chi_i(g)u_i$ for all $g\in G$ and $i\in\{1,\ldots,d\}$.
The following are equivalent:
\begin{itemize}
\item[(a)]
$G$ acts freely on $S(\mathbb{C}^d)$.
\item[(b)]
For each $i\in\{1,\ldots,d\}$, the character $\chi_i$ is an isomorphism of $G$ onto its image. 
\item[(c)]
For each $i,j\in\{1,\ldots,d\}$, there exists a smallest integer $m_{ij}\geq0$ for which $\chi_i\chi_j^{m_{ij}}=1$.
\end{itemize}
Furthermore, when (c) holds, the polynomial map $f\colon \mathbb{C}^d\to\mathbb{C}^{d\times d}$ defined by
\[
f(x)=\{\langle u_i,x\rangle \langle u_j,x\rangle^{m_{ij}}\}_{i,j=1}^d
\]
is $G$-invariant and $f^\downarrow|_{S(\mathbb{C}^d)/G}$ is bilipschitz.
\end{theorem}

Notably, (the proof of) Proposition~5.2.1 in~\cite{Dufresne:08} establishes that the related map 
\[
p(x)=\{\langle u_i,x\rangle \langle u_j,x\rangle^{m_{ij}}\}_{i\leq j}
\]
is $G$-invariant and $p^\downarrow$ is injective.
Later, \cite{CahillCC:20} established that while $p^\downarrow|_{S(\mathbb{C}^d)/G}$ is Lipschitz, when $d\geq3$ and $|G|\geq3$, this function is \textit{not} lower Lipschitz.
With this context, Theorem~\ref{thm.constructive invariants in abelian case} can be interpreted as making the injective map $p^\downarrow|_{S(\mathbb{C}^d)/G}$ lower Lipschitz by including additional coordinate functions for $i>j$.

\begin{proof}[Proof of Theorem~\ref{thm.constructive invariants in abelian case}]
For (a)$\Rightarrow$(b), suppose $G$ acts freely on $S(\mathbb{C}^d)$.
Then $1$ is not an eigenvalue of any nonidentity $g\in G$.
It follows that for each $i$, the character $\chi_i$ has trivial kernel.
Then the first isomorphism theorem gives $G\cong \operatorname{im}(\chi_i)$.

For (b)$\Rightarrow$(c), observe that $G$ is cyclic, and given a generator $h$ of $G$, then for each $i$, $\chi_i(h)$ is a primitive $|G|$th root of unity.
Fix $i$ and $j$.
Then there exists an integer $m\geq0$ such $\chi_j(h)^{m} = \chi_i(h)^{-1}$, and so $\chi_i(h^k)\chi_j(h^k)^{m}=(\chi_i(h)\chi_j(h)^{m})^k=1$ for every $k$.
The claim then follows from the least integer principle.

For (c)$\Rightarrow$(a), consider any $g\in G$ for which there exists $x\in S(\mathbb{C}^d)$ such that $gx=x$.
Then $1$ is an eigenvalue of $g$, meaning there exists $j$ such that $\chi_j(g)=1$.
Our assumption then gives $\chi_i(g)=\chi_i(g)\chi_j(g)^{m_{ij}}=1$ for all $i$.
Thus, $g=\operatorname{id}$, i.e., $G$ acts freely on $S(\mathbb{C}^d)$.

Finally, suppose (c) holds.
Then for every $g\in G$, we have
\begin{align*}
f(gx)_{ij}
&=\langle u_i,gx\rangle\langle u_j,gx\rangle^{m_{ij}}
=\langle g^{-1}u_i,x\rangle\langle g^{-1}u_j,x\rangle^{m_{ij}}\\
&=\langle \overline{\chi_i(g)}u_i,x\rangle \langle\overline{\chi_j(g)}u_j,x\rangle^{m_{ij}}
=\chi_i(g)\chi_j(g)^{m_{ij}}\cdot\langle u_i,x\rangle \langle u_j,x\rangle^{m_{ij}}
=f(x)_{ij}.
\end{align*}
Thus, $f$ is $G$-invariant.
It remains to show that $f^\downarrow|_{S(\mathbb{C}^d)/G}$ is bilipschitz.
By Propositions~\ref{prop.quotient immersion} and~\ref{prop.embedding is bilipschitz}, it suffices to show that $f^\downarrow|_{S(\mathbb{C}^d)/G}$ is injective and $f|_{S(\mathbb{C}^d)}$ is an immersion.
Injectivity follows from (the proof of) Proposition~5.2.1 in~\cite{Dufresne:08}.
To prove immersion, assume $\{u_i\}_{i=1}^d$ is the standard basis without loss of generality so that $f(x)=\{x_i x_j^{m_{ij}}\}_{i,j=1}^d$.
Fix $x\in S(\mathbb{C}^d)$, select an index $j$ at which $x_{j}\neq0$, and let $A$ denote the $d\times d$ submatrix of the (complex) Jacobian $Df(x)$ corresponding to row indices $\{(i,j)\}_{i=1}^d$.
Then $A=\operatorname{diag}u+ve_{j}^\top$, where
\[
u=\left[\begin{array}{c}
x_{j}^{m_{1j}}\\
\vdots\\
x_{j}^{m_{dj}}
\end{array}\right],
\qquad
v=\left[\begin{array}{c}
m_{1j}x_1x_{j}^{m_{1j}-1}\\
\vdots\\
m_{dj}x_dx_{j}^{m_{dj}-1}
\end{array}\right].
\]
The matrix determinant lemma gives
\[
\operatorname{det}(A)
=\big(1+e_{j}^\top(\operatorname{diag}u)^{-1}v\big)\cdot\operatorname{det}(\operatorname{diag}u)
=(1+m_{jj})\cdot x_j^{\sum_i m_{ij}}
\neq0,
\]
and so $Df(x)$ is injective.
Then the corresponding real Jacobian $\mathbb{R}^{2d}\to\mathbb{R}^{2d^2}$ is also injective, as desired.
\end{proof}

\section{Maps of minimum distortion}
\label{sec.minimum distortion}

The \textbf{Euclidean distortion} of a metric space $X$, denoted by $c_2(X)\in[1,\infty]$, is the infimum of $c$ for which there exists a Hilbert space $H$ and a bilipschitz map $X\to H$ of distortion $c$.
In particular, $c_2(X)<\infty$ if and only if $X$ admits a bilipschitz embedding into some Hilbert space.
Euclidean distortion is finitely determined:

\begin{proposition}
\label{prop.euclidean distortion is finitely determined}
It holds that $c_2(X)=\sup\{c_2(F):F\subseteq X,|F|<\infty\}$.
\end{proposition}

\begin{proof}
The inequality $\geq$ holds since restricting a bilipschitz function over $X$ to $F$ can only improve the bilipschitz constants.
For the other direction, suppose that for every finite $F\subseteq X$ and every $\epsilon>0$, one can embed $F$ into a Hilbert space with distortion at most $c+\epsilon$.
Then the span of the image of this embedding has dimension at most $|F|<\infty$, and so the Hilbert space can be taken to be $\ell^2$ without loss of generality.
Then $X$ embeds into an ultra power $H$ of $\ell^2$ with distortion at most $c$ {by Theorem~4.6 in~\cite{GarciaG:22}}.
Furthermore, $H$ is a Hilbert space by Theorem~3.3(ii) in~\cite{Heinrich:80}.
\end{proof}

In fact, Proposition~\ref{prop.euclidean distortion is finitely determined} can be strengthened as follows:

\begin{lemma}
\label{lem.euclidean distortion dense}
Given a dense subset $Y$ of a metric space $X$, it holds that
\[
c_2(X)=\sup\{c_2(F):F\subseteq Y,|F|<\infty\}.
\]
\end{lemma}

\begin{proof}
By Proposition~\ref{prop.euclidean distortion is finitely determined}, it suffices to prove
\[
\sup\{c_2(E):E\subseteq X,|F|<\infty\}
=\sup\{c_2(F):F\subseteq Y,|F|<\infty\}.
\]
The inequality $\geq$ follows from the containment $Y\subseteq X$.
For the reverse inequality, fix $\gamma>1$.
Since $Y$ is dense in $X$, for each $E:=\{x_1,\ldots,x_m\}\subseteq X$, we may select $F:=\{y_1,\ldots,y_m\}\subseteq Y$ close enough to $E$ so that $g\colon E\to F$ defined by $g\colon x_i\mapsto y_i$ has distortion at most $\gamma$.
Furthermore, there exists $f\colon F\to \ell^2$ of distortion at most $\gamma \cdot c_2(F)$.
Then
\[
c_2(E)
\leq\operatorname{dist}(f\circ g)
\leq\operatorname{dist}(f)\operatorname{dist}(g)
\leq\gamma^2 \cdot c_2(F).
\]
The desired bound follows by taking the supremum of both sides and recalling that $\gamma>1$ was arbitrary.
\end{proof}

Proposition~\ref{prop.euclidean distortion is finitely determined} focuses our attention to finite metric spaces, in which case~\cite{LinialLR:95} observed that Euclidean distortion can be computed by semidefinite programming:

\begin{proposition}
\label{prop.sdp}
Given a finite metric space $X$, then $c_2(X)^2$ is the infimum of $t$ for which there exists a positive semidefinite matrix $Q\in\mathbb{R}^{X\times X}$ such that
\[
d_X(x,y)^2
\leq Q_{xx}-2Q_{xy}+Q_{yy}
\leq t\cdot d_X(x,y)^2
\qquad
\forall x,y\in X.
\]
\end{proposition}

\begin{proof}[Proof sketch]
Each embedding $f\colon X\to H$ determines $Q\succeq0$ defined by $Q_{xy}=\langle f(x),f(y)\rangle$, in which case $\|f(x)-f(y)\|^2=Q_{xx}-2Q_{xy}+Q_{yy}$.
Conversely, each $Q\succeq0$ determines an embedding $f\colon X\to \mathbb{R}^X$ up to post-composition by an orthogonal transformation.
\end{proof}

One may apply weak duality, as was done in~\cite{LinialLR:95}, to obtain the following result.
Here, we let $D_X\in\mathbb{R}^{X\times X}$ denote the matrix defined by $(D_X)_{xy}=d_X(x,y)^2$, while $Q_+$ and $Q_-$ denote the entrywise positive and negative parts of a matrix $Q$, respectively.

\begin{proposition}[Corollary~3.5 in~\cite{LinialLR:95}]
\label{prop.weak duality}
Given a finite metric space $X$, consider any bilipschitz map $f\colon X\to\ell^2$ and any positive semidefinite $Q\in\mathbb{R}^{X\times X}$ such that $Q1=0$.
Then
\[
\langle D_X,Q_+\rangle
\leq \operatorname{dist}(f)^2 \cdot \langle D_X,Q_-\rangle.
\]
Furthermore, if equality holds with $Q \neq 0$, then $\operatorname{dist}(f) = c_2(X)$.
\end{proposition}

We take inspiration from the proof of Claim~2.1 in~\cite{LinialM:00} to prove the following:

\begin{theorem}
\label{thm.euclidean distortion of circle}
Take $X=\mathbb{R}/\mathbb{Z}$ with any translation-invariant metric $d$ for which the function
\[
g\colon (0,\tfrac{1}{2}]\to\mathbb{R},
\qquad
g(t)=\frac{|e^{2\pi i t}-1|}{d(t+\mathbb{Z},\mathbb{Z})}
\]
is monotonically decreasing.
Then $f\colon X\to\mathbb{C}$ defined by $f(t+\mathbb{Z})=e^{2\pi it}$ has distortion
\begin{equation}
\label{eq.distortion of circle}
\operatorname{dist}(f)
=\lim_{t\to0}\frac{g(t)}{g(\frac{1}{2})}
=c_2(X)
\leq\frac{\pi}{2}.
\end{equation}
\end{theorem}

\begin{proof}
{First, we show the first equality.
Take any $A,B\in X$ with $A\neq B$.
After swapping $A$ and $B$ as necessary, we may express $A=s+\mathbb{Z}$ and $B=t+\mathbb{Z}$ for some $s,t\in\mathbb{R}$ with $s-t\in(0,\frac{1}{2}]$.
Then
\[
\frac{|f(A)-f(B)|}{d(A,B)}
=\frac{|f(s+\mathbb{Z})-f(t+\mathbb{Z})|}{d(s+\mathbb{Z},t+\mathbb{Z})}
=\frac{|e^{2\pi i(s-t)}-1|}{d(s-t+\mathbb{Z},\mathbb{Z})}
=g(s-t).
\]
Since $g$ is monotonically decreasing, it follows that $f$ has optimal bilipschitz bounds
\[
\alpha
=g(\tfrac{1}{2})
\qquad
\text{and}
\qquad
\beta
=\lim_{t\to0}g(t)\in(0,\infty].
\]
Next,} given an even number $n\in\mathbb{N}$, the cyclic group $C:=\mathbb{Z}/n\mathbb{Z}$ enjoys {an} injective homomorphism $\phi\colon C \to X$, through which we may pull back $d$ to endow $C$ with a metric $\tilde{d}$.
Explicitly,
\[
\phi(k+n\mathbb{Z})
=\tfrac{k}{n}+\mathbb{Z}
\qquad
\text{and}
\qquad
\tilde{d}(k+n\mathbb{Z},\ell+n\mathbb{Z})
=d(\tfrac{k-\ell}{n}+\mathbb{Z},\mathbb{Z}).
\]
Then the optimal bilipschitz bounds of $f\circ\phi$ are obtained by minimizing and maximizing the quantity
\[
\frac{|(f\circ \phi)(k+n\mathbb{Z})-(f\circ \phi)(\ell+n\mathbb{Z})|}{\tilde{d}(k+n\mathbb{Z},\ell+n\mathbb{Z})}
=\frac{|e^{2\pi i(\frac{k-\ell}{n})}-1|}{d(\tfrac{k-\ell}{n}+\mathbb{Z},\mathbb{Z})}.
\]
In particular, our assumptions on $g$ and the fact that $n$ is even together imply that $f\circ\phi$ has optimal bilipschitz bounds
\[
g(\tfrac{1}{2})
=\frac{2}{d(\frac{1}{2}+\mathbb{Z},\mathbb{Z})},
\qquad
g(\tfrac{1}{n})
=\frac{2\sin(\frac{\pi}{n})}{d(\frac{1}{n}+\mathbb{Z},\mathbb{Z})}.
\]
{Furthermore, the triangle inequality gives $d(\frac{1}{2}+\mathbb{Z},\mathbb{Z})\leq\frac{n}{2}\cdot d(\frac{1}{n}+\mathbb{Z},\mathbb{Z})$, and so
\[
\frac{g(\frac{1}{n})}{g(\frac{1}{2})}
=\sin(\tfrac{\pi}{n})\cdot\frac{d(\frac{1}{2}+\mathbb{Z},\mathbb{Z})}{d(\frac{1}{n}+\mathbb{Z},\mathbb{Z})}
\leq\sin(\tfrac{\pi}{n})\cdot\tfrac{n}{2}
\leq\tfrac{\pi}{2}.
\]
Since $g$ is monotonically decreasing, we conclude that the limit in~\eqref{eq.distortion of circle} is at most $\frac{\pi}{2}$.
}

By identifying $C\cong\{0,\ldots,n-1\}$ with arithmetic modulo $n$, we define $Q\in\mathbb{R}^{C\times C}$ by
\[
Q_{ij}=
\left\{\begin{array}{cl}
2\cos^2(\tfrac{\pi}{n})&\text{if $i=j$}\\
-1&\text{if $i=j\pm 1$}\\
2\sin^2(\tfrac{\pi}{n})&\text{if $i=j+\frac{n}{2}$}\\
0&\text{otherwise.}
\end{array}\right.
\]
Then $Q1=0$ and $Q$ is positive semidefinite since its eigenvalues are nonnegative (this is verified in Claim~2.3 of~\cite{LinialM:00}).
Furthermore, since $D_C$ and $Q$ are both circulant, we have
\[
\langle D_C,Q_+\rangle
=n\cdot d(\tfrac{1}{2}+\mathbb{Z},\mathbb{Z})^2\cdot 2\sin^2(\tfrac{\pi}{n}),
\qquad
\langle D_C,Q_-\rangle
=n\cdot d(\tfrac{1}{n}+\mathbb{Z},\mathbb{Z})^2\cdot 2,
\]
and so
\[
\frac{\langle D_C,Q_+\rangle}{\langle D_C,Q_-\rangle}
=\frac{d(\tfrac{1}{2}+\mathbb{Z},\mathbb{Z})^2}{d(\tfrac{1}{n}+\mathbb{Z},\mathbb{Z})^2}\cdot\sin^2(\tfrac{\pi}{n})
=\bigg(\frac{g(\frac{1}{n})}{g(\frac{1}{2})}\bigg)^2
=\operatorname{dist}(f\circ\phi)^2.
\]
By Proposition~\ref{prop.weak duality}, it follows that $\operatorname{dist}(f\circ\phi)=c_2(C)$.
Since $\phi$ is an isometric embedding, the order-$n$ subgroup $X_n$ of $X$ also has Euclidean distortion $g(\frac{1}{n})/g(\frac{1}{2})$.
Overall, we have
\[
\lim_{t\to0}\frac{g(t)}{g(\frac{1}{2})}
=\frac{\beta}{\alpha}
=\operatorname{dist}(f)
\geq
c_2(X)
\geq\sup_{\substack{n\in\mathbb{N}\\n\text{ even}}} c_2(X_n)
=\sup_{\substack{n\in\mathbb{N}\\n\text{ even}}} \frac{g(\frac{1}{n})}{g(\frac{1}{2})}
=\lim_{t\to0}\frac{g(t)}{g(\frac{1}{2})},
\]
which implies the result.
\end{proof}

We may use Theorem~\ref{thm.euclidean distortion of circle} to establish that the embeddings in Section~\ref{sec:homogeneous extension} achieve the minimum possible distortion:

\begin{corollary}
\label{cor.c2 real phase retrieval}
Given a real Hilbert space $V$ of dimension at least $2$, consider the group $G:=\{ \pm\operatorname{id}\}\leq O(V)$.
Then $c_2(V/G)=\sqrt{2}$.
\end{corollary}

\begin{proof}
Denote $X=\mathbb{R}/\mathbb{Z}$, select orthonormal $u,v\in V$, and define $\phi\colon X\to V/G$ by
\[
\phi(t+\mathbb{Z})
=[ u\cos(\pi t)+v\sin(\pi t) ].
\]
Then $\phi$ is injective.
Let $d$ denote the pullback of the quotient metric on $V/G$.
Then for $t\in(0,\frac{1}{2}]$, we have
\[
d(t+\mathbb{Z},\mathbb{Z})^2
=\min\{\|\phi(t+\mathbb{Z})-\phi(\mathbb{Z})\|^2,\|\phi(t+\mathbb{Z})+\phi(\mathbb{Z})\|^2\}
=2-2\cos(\pi t)
\]
and $|e^{2\pi it}-1|^2=4-4\cos^2(\pi t)$, and so
\[
g(t)
:=\frac{|e^{2\pi i t}-1|}{d(t+\mathbb{Z},\mathbb{Z})}
=\sqrt{2\cdot\frac{1-\cos^2(\pi t)}{1-\cos(\pi t)}}
=\sqrt{2\cdot(1+\cos(\pi t))}.
\]
Since $g$ is decreasing in $t$, Theorem~\ref{thm.euclidean distortion of circle} gives that
\[
c_2(X)
=\lim_{t\to0}\frac{g(t)}{g(\frac{1}{2})}
=\sqrt{2}.
\]
Finally, take $f^\star$ as constructed in Example~\ref{ex.real phase retrieval}.
Then
\[
\sqrt{2}
= c_2(X)
= c_2(\operatorname{im}\phi)
\leq c_2(V/G)
\leq \operatorname{dist}(f^\star)
=\sqrt{2},
\]
which implies the result.
\end{proof}

A nearly identical proof gives the following:

\begin{corollary}
\label{cor.c2 complex phase retrieval}
Given a complex Hilbert space $V$ of dimension at least $2$, consider the group $G:=\{\omega\cdot\operatorname{id}:|\omega|=1\}\leq U(V)$.
Then $c_2(V/G)=\sqrt{2}$.
\end{corollary}

\begin{corollary}
\label{cor.c2 discrete rotations}
Given $V:=\mathbb{C}$, consider the group $G:=\langle e^{2\pi i/r}\rangle\leq U(1)$ for some $r\in\mathbb{N}$.
Then $c_2(V/G)=r\sin(\frac{\pi}{2r})$.
\end{corollary}

\begin{proof}
Denote $X=\mathbb{R}/\mathbb{Z}$, and define $\phi\colon X\to V/G$ by $\phi(t+\mathbb{Z})=[e^{2\pi i t/r}]$.
Let $d$ denote the pullback of the quotient metric on $V/G$.
One may show the function $g$ from Theorem~\ref{thm.euclidean distortion of circle} is given by $g(t)=\frac{\sin(\pi t)}{\sin(\pi t/r)}$, which is decreasing by a calculus-based argument like the one in Example~\ref{ex.discrete rotations}.
Then Theorem~\ref{thm.euclidean distortion of circle} gives that $c_2(X)$ matches the distortion of the homogeneous extension of (a lift of) the map $f$ from Example~\ref{ex.discrete rotations}.
The result follows.
\end{proof}

\begin{lemma}
\label{lem.cartesian product embeddings}
Given nontrivial real Hilbert spaces $V_1$ and $V_2$ and groups $G_1\leq O(V_1)$ and $G_2\leq O(V_2)$, it holds that
\[
c_2\Big((V_1\oplus V_2)\dslash (G_1\times G_2)\Big)
=\max\Big\{c_2(V_1\dslash G_1),c_2(V_2\dslash G_2)\Big\}.
\]
\end{lemma}

\begin{proof}
First, we prove that $\geq$ holds.
Given a Hilbert space $H$ and a $(G_1\times G_2)$-invariant map $f\colon V_1\oplus V_2\to H$, we have that $f_1:=f(\cdot,0)\colon V_1\to H$ is $G_1$-invariant and $f_2:=f(0,\cdot)\colon V_2\to H$ is $G_2$-invariant.
For each $i\in\{1,2\}$, we may identify $f_i^\downarrow\colon V_i\dslash G_i\to H$ with a restriction of $f^\downarrow\colon (V_1\oplus V_2)\dslash (G_1\times G_2)\to H$, and so $\operatorname{dist}(f_i^\downarrow)\leq\operatorname{dist}(f^\downarrow)$.
The desired bound follows.

Next, we prove that $\leq$ holds.
If the right-hand side is infinite, then we are done.
Otherwise, for each $i\in\{1,2\}$, there exists a Hilbert space $H_i$ and a $G_i$-invariant map $f_i\colon V_i\to H_i$ of finite distortion.
Scale as necessary so that the optimal bilipschitz bounds for $f_i^\downarrow$ equal $1$ and $\beta_i$, and define $f\colon V_1\oplus V_2\to H_1\oplus H_2$ by $f(v_1,v_2)=(f_1(v_1),f_2(v_2))$.
Put $\beta:=\max\{\beta_1,\beta_2\}$.
Then $f$ is $(G_1\times G_2)$-invariant and
\begin{align*}
\|f^\downarrow([(u_1,u_2)])-f^\downarrow([(v_1,v_2)])\|_{H_1\oplus H_2}^2
&=\|f_1^\downarrow([u_1])-f_1^\downarrow([v_1])\|_{H_1}^2+\|f_2^\downarrow([u_2])-f_2^\downarrow([v_2])\|_{H_2}^2\\
&\leq\beta_1^2 \cdot d_{V_1\dslash G_1}([u_1],[v_1])^2+\beta_2^2 \cdot d_{V_2\dslash G_2}([u_2],[v_2])^2\\
&\leq\beta^2\cdot\Big(d_{V_1\dslash G_1}([u_1],[v_1])^2+d_{V_2\dslash G_2}([u_2],[v_2])^2\Big)\\
&=\beta^2\cdot d_{(V_1\oplus V_2)\dslash (G_1\times G_2)}([(u_1,u_2)],[(v_1,v_2)])^2,
\end{align*}
and similarly, $f^\downarrow$ has lower Lipschitz bound $1$.
Thus, $\operatorname{dist}(f^\downarrow)\leq \beta=\max\{\operatorname{dist}(f_1^\downarrow),\operatorname{dist}(f_2^\downarrow)\}$.
The desired bound follows.
\end{proof}

This result is similar in spirit to Lemma~3.2 in~\cite{ErikssonBique:18}.
We can use Lemma~\ref{lem.cartesian product embeddings} (and its proof) to produce more examples of embeddings of minimum distortion.
For example, we may express $\mathbb{R}^2/\langle [\begin{smallmatrix}1&\phantom{-}0\\0&-1\end{smallmatrix}]\rangle$ as $(\mathbb{R}\oplus\mathbb{R})/(\{\operatorname{id}\}\times\{\pm\operatorname{id}\})$.
This suggests taking $f_1\colon x\mapsto x$ and $f_2\colon y\mapsto |y|$ so that $f\colon(x,y)\mapsto (x,|y|)$.
In fact, the resulting map $f^\downarrow$ isometrically embeds $\mathbb{R}^2/\langle [\begin{smallmatrix}1&\phantom{-}0\\0&-1\end{smallmatrix}]\rangle$ into $\mathbb{R}^2$.
Unlike the previous examples, this distortion-minimizing map is not the homogeneous extension of a polynomial on the sphere, although it is positively homogeneous.

\section{Quotients by permutation}
\label{sec.permutation}

In this section, we consider the real Hilbert space $V:=\ell^2(\mathbb{N};\mathbb{R}^d)$ of sequences of vectors in $\mathbb{R}^d$ whose norms are square summable.
The group $G:=S_\infty$ of bijections $\mathbb{N}\to\mathbb{N}$ acts on the index set of these sequences, and we are interested in the metric space $V\dslash G$.
As indicated in Example~\ref{ex.ell2 with permutations}, the orbits in this example are not closed, and so unlike the other examples we have considered, this metric quotient does not simply reduce to the honest quotient $V/G$.
We will see that this space has an isometric embedding into $\ell^2$ when $d=1$, but has no bilipschitz embedding into any Hilbert space when $d\geq 3$.
We start by characterizing the members of the metric quotient:

\begin{lemma}
\label{lem.orbit closures from S_infty}
Consider any $x\in\ell^2(\mathbb{N};\mathbb{R}^d)$.
Then $[x]\in\ell^2(\mathbb{N};\mathbb{R}^d)\dslash S_\infty$ is given by
\begin{equation}
\label{eq.bracket mod S infty}
[x]=\Big\{ y\in {\ell^2(\mathbb{N};\mathbb{R}^d)} : \exists \text{ bijection } \sigma\colon\operatorname{supp}(x)\to\operatorname{supp}(y), ~ \forall i\in\operatorname{supp}(y), ~ y_i=x_{\sigma^{-1}(i)} \Big\}.
\end{equation}
\end{lemma}

\begin{proof}
Fix $x\in\ell^2(\mathbb{N};\mathbb{R}^d)$, and let $A_x$ denote the right-hand side of \eqref{eq.bracket mod S infty}.
By Lemma~\ref{lem.metric quotient orbit closure}, $[x]=\overline{S_\infty\cdot x}$.
As such, our task is to show that $A_x=\overline{S_\infty\cdot x}$.

First, we show that $A_x\subseteq\overline{S_\infty\cdot x}$.
Fix $y\in A_x$, and let $\sigma$ be as in \eqref{eq.bracket mod S infty}.
Let $i_1<i_2<\cdots$ denote the (possibly finite) list of members of $\operatorname{supp}(x)$.
For each $n\in\mathbb{N}$, consider the index set $I_n:=\{i_j:j\leq n\}$, select any $\pi_n\in S_\infty$ such that $\pi_n(i)=\sigma(i)$ for every $i\in I_n$, and put $p_n:=x\circ \pi_n^{-1}$.
Then $p_n\in S_\infty\cdot x$ for every $n$, and we claim that $p_n\to y$.
Indeed, the bound
\[
\|a-b\|^2
=\|a\|^2-2\langle a,b\rangle+\|b\|^2
\leq\|a\|^2+2\|a\|\|b\|+\|b\|^2
\leq {2}\|a\|^2+{2}\|b\|^2
\]
for $a,b\in\mathbb{R}^d$ implies
\[
\|p_n-y\|_{\ell^2}^2
=\sum_{i\not\in I_n} \|x_i-y_{\pi_n(i)}\|^2
\leq {2}\sum_{i\not\in  I_n} \|x_i\|^2+{2}\sum_{i\not\in I_n} \|y_{\pi_n(i)}\|^2
={2}\sum_{j>n} \|x_{i_j}\|^2+{2}\sum_{j>n} \|y_{\sigma(i_j)}\|^2,
\]
which vanishes as $n\to\infty$.
Thus, $y\in\overline{S_\infty\cdot x}$.

Next, we show that $\overline{S_\infty\cdot x}\subseteq A_x$.
Fix $y\in \overline{S_\infty\cdot x}$.
Then there exists $\{p_n:=x\circ \pi_n^{-1}\}_{n=1}^\infty$ in $S_\infty\cdot x$ that converges to $y$.
For each $s>0$, consider the threshold function $\theta_s\colon \ell^2(\mathbb{N};\mathbb{R}^d)\to\ell^2(\mathbb{N};\mathbb{R}^d)$ defined by
\[
\theta_s(z)_i:=\left\{\begin{array}{cl} z_i&\text{if }\|z_i\|\geq s\\0&\text{otherwise,}\end{array}\right.
\]
and denote $t(s):=\|x-\theta_s(x)\|_\infty$.
Then either $t(s)=0$ or $t(s)>0$, in which case
\[
t(s)
=\sup\{\|x_i\|:\|x_i\|<s\}
=\sup\{\|x_i\|:t(s)/2<\|x_i\|<s\}
<s,
\]
since $x\in\ell^2(\mathbb{N};\mathbb{R})$ has finitely many terms $x_i$ with $\|x_i\|>t(s)/2$.
Either way, $t(s)<s$.
In fact, for every $g\in S_\infty$, we have $\theta_s(gx)=g\theta_s(x)$, and so $t(s)=\|gx-\theta_s(gx)\|_\infty<s$.
Thus, given $g,h\in S_\infty$, if $\operatorname{supp}\theta_s(gx)\neq\operatorname{supp}\theta_s(hx)$, then $\|gx-hx\|_{\ell^2}\geq s-t(s)>0$.
Also, if $\operatorname{supp}\theta_s(gx)=\operatorname{supp}\theta_s(hx)$ but $\theta_s(gx)\neq \theta_s(hx)$, then $\|gx-hx\|_{\ell^2}\geq\|\theta_s(gx)-\theta_s(hx)\|_{\ell^2}$ is bounded away from zero, too.
Since $\{p_n\}_{n=1}^\infty$ is Cauchy, it follows that $\theta_s(p_m)=\theta_s(p_n)$ for all sufficiently large $m$ and $n$.
That is, for each $s>0$, we have $\theta_s(p_n)=\theta_s(y)$ for all $n\geq N(s)$.
From this information, one may construct $\sigma\colon\operatorname{supp}(x)\to\operatorname{supp}(y)$, for example, $\sigma(i)=\pi_{N(\lceil \|x_i\|^{-1}\rceil^{-1})}(i)$, such that $y_i=x_{\sigma^{-1}(i)}$ for all $i\in\operatorname{supp}(y)$.
Thus, $y\in A_x$.
\end{proof}

As a consequence of Lemma~\ref{lem.orbit closures from S_infty}, we can establish that there is no orthogonal group $G$ such that $\ell^2(\mathbb{N};\mathbb{R})\dslash S_\infty=\ell^2(\mathbb{N};\mathbb{R})/G$.
Indeed, let $e_i$ denote the $i$th standard basis element in $\ell^2(\mathbb{N};\mathbb{R})$.
Then $[e_1]=\{e_i:i\in\mathbb{N}\}$.
Thus, given $G$ such that $G\cdot e_1 = [e_1]$, it must hold that every element of $G$ permutes the standard basis, i.e., $G\leq S_\infty$.
But then every member of the $G$-orbit of $x:=\{\frac{1}{n}\}_{n=1}^\infty$ is entrywise nonnegative, whereas $[x]$ contains members (such as the forward shift of $x$) with entries equal to $0$.

In what follows, we construct an isometric embedding of $\ell^2(\mathbb{N};\mathbb{R})\dslash S_\infty$ in $\ell^2(\mathbb{N};\mathbb{R})$.
{Given $x\in\ell^2(\mathbb{N};\mathbb{R})$,} let $x_+$ and $x_-$ denote the positive and negative parts of $x$ (respectively) so that $x=x_+-x_-$, {and in the case where $x$ is entrywise nonnegative,} let $x_{(n)}$ denote its $n$th largest entry.

\begin{theorem}
\label{thm.embedding for d=1}
Consider $\Phi\colon \ell^2(\mathbb{N};\mathbb{R})\dslash S_\infty \to\ell^2(\mathbb{N};\mathbb{R})$ defined by
\[
\Phi([x])_n
=\left\{\begin{array}{cl}
(x_+)_{(\frac{n}{2})} &\text{if $n$ is even}\\
-(x_-)_{(\frac{n+1}{2})} &\text{if $n$ is odd.}
\end{array}\right.
\]
Then $\|\Phi([x])-\Phi([y])\|_{\ell^2}=d_{\ell^2(\mathbb{N})\dslash S_\infty}([x],[y])$.
Consequently, $c_2(\ell^2(\mathbb{N};\mathbb{R})\dslash S_\infty)=1$.
\end{theorem}

The following example illustrates how $\Phi$ sorts the positive and negative parts of $x$:
\begin{align*}
x
&=(\phantom{-}1,\phantom{-}\tfrac{1}{2},-\tfrac{1}{3},-\tfrac{1}{4},\phantom{-}\tfrac{1}{5},\phantom{-}\tfrac{1}{6},-\tfrac{1}{7},-\tfrac{1}{8},\phantom{-}\ldots\phantom{-})\\
\Phi([x])
&=(-\tfrac{1}{3},\phantom{-}1,-\tfrac{1}{4},\phantom{-}\tfrac{1}{2},-\tfrac{1}{7},\phantom{-}\tfrac{1}{5},-\tfrac{1}{8},\phantom{-}\tfrac{1}{6},\phantom{-}\ldots\phantom{-}).
\end{align*}

\begin{proof}[Proof of Theorem~\ref{thm.embedding for d=1}]
Fix $x,y\in\ell^2(\mathbb{N};\mathbb{R})$.
Since $\Phi([x])\in[x]$ and $\Phi([y])\in[y]$, it holds that
\[
\|\Phi([x])-\Phi([y])\|_{\ell^2}
\geq d_{\ell^2(\mathbb{N})\dslash S_\infty}([x],[y]).
\]
For the reverse inequality, consider any $p_0,q_0\in\ell^2(\mathbb{N};\mathbb{R})$.
We claim there exists a sequence $\{p_n\}_{n=1}^\infty$ in $[p_0]$ and a sequence $\{q_n\}_{n=1}^\infty$ in $[q_0]$ such that $p_n\to\Phi([p_0])$, $q_n\to\Phi([p_0])$, and
\[
\|p_{n+1}-q_{n+1}\|_{\ell^2}
\leq \|p_n-q_n\|_{\ell^2}
\qquad
\forall n\geq0.
\]
From this, it follows that $\|\Phi([x])-\Phi([y])\|_{\ell^2}\leq \|p_0-q_0\|_{\ell^2}$ for every $p_0\in[x]$ and $q_0\in[y]$, and taking the infimum of both sides gives
\[
\|\Phi([x])-\Phi([y])\|_{\ell^2}
\leq
d_{\ell^2(\mathbb{N})\dslash S_\infty}([x],[y]),
\]
as desired.

We construct $\{p_n\}_{n=1}^\infty$ and $\{q_n\}_{n=1}^\infty$ so as to satisfy
\begin{equation}
\label{eq.equal in initial tuple}
(p_n)_{[1:n]}=\Phi([p_0])_{[1:n]}
\qquad
\text{and}
\qquad
(q_n)_{[1:n]}=\Phi([q_0])_{[1:n]}
\qquad
\forall n\in\mathbb{N}.
\end{equation}
(Given a sequence $x$ and natural numbers $s<t$, we write $x_{[s:t]}:=(x_s,\ldots,x_t)$.)
In what follows, we describe the map $(p_n,q_n)\mapsto(p_{n+1},q_{n+1})$.
Take
\[
a:=\{(-1)^{{n+1}}(p_n)_i\}_{i=n+1}^\infty
\qquad
\text{and}
\qquad
b:=\{(-1)^{{n+1}}(q_n)_i\}_{i=n+1}^\infty.
\]
Since $a,b\in\ell^2(\mathbb{N};\mathbb{R})$, we have $\sup a\geq0$ and $\sup b\geq0$. 
If $\sup a=0$ and $\sup b=0$, put
\begin{equation}
\label{eq.iter case 1}
\left[\begin{array}{c}
p_{n+1} \\
q_{n+1}
\end{array}\right]
=\left[\begin{array}{ccc}
(p_n)_{[1:n]} & 0 & (p_n)_{[n+1:\infty)}\\
(q_n)_{[1:n]} & 0 & (q_n)_{[n+1:\infty)}
\end{array}\right].
\end{equation}
(Here and below, we express $p_{n+1}$ and $q_{n+1}$ as row vectors of a $2\times\infty$ matrix for simplicity.)
If $\sup a>0$ and $\sup b=0$, then we take any $i\in\arg\max(a)$ and put
\begin{equation}
\label{eq.iter case 2}
\left[\begin{array}{c}
p_{n+1} \\
q_{n+1}
\end{array}\right]
=\left[\begin{array}{ccccc}
(p_n)_{[1:n]} & (p_n)_{n+i} & (p_n)_{[n+1:n+i-1]} & 0 & (p_n)_{[n+i+1:\infty)} \\
(q_n)_{[1:n]} & 0 & (q_n)_{[n+1:n+i-1]} & (q_n)_{n+i} & (q_n)_{[n+i+1:\infty)}
\end{array}\right].
\end{equation}
Similarly, if $\sup a=0$ and $\sup b>0$, then we take any $i\in\arg\max(b)$ and put
\begin{equation}
\label{eq.iter case 3}
\left[\begin{array}{c}
p_{n+1} \\
q_{n+1}
\end{array}\right]
=\left[\begin{array}{ccccc}
(p_n)_{[1:n]} & 0 & (p_n)_{[n+1:n+i-1]} & (p_n)_{n+i} & (p_n)_{[n+i+1:\infty)} \\
(q_n)_{[1:n]} & (q_n)_{n+i} & (q_n)_{[n+1:n+i-1]} & 0 & (q_n)_{[n+i+1:\infty)}
\end{array}\right].
\end{equation}
Finally, if $\sup a>0$ and $\sup b>0$, then we take any $i\in\arg\max(a)$ and $j\in\arg\max(b)$.
If $i=j$, then we put
\begin{equation}
\label{eq.iter case 4}
\left[\begin{array}{c}
p_{n+1} \\
q_{n+1}
\end{array}\right]
=\left[\begin{array}{cccc}
(p_n)_{[1:n]} & (p_n)_{n+i} & (p_n)_{[n+1:n+i-1]} & (p_n)_{[n+i+1:\infty)} \\
(q_n)_{[1:n]} & (q_n)_{n+i} & (q_n)_{[n+1:n+i-1]} & (q_n)_{[n+i+1:\infty)}
\end{array}\right].
\end{equation}
Otherwise, if $i\neq j$, then we put
\begin{equation}
\label{eq.iter case 5}
\left[\begin{array}{c}
p_{n+1} \\
q_{n+1}
\end{array}\right]
=\left[\begin{array}{cccc}
(p_n)_{[1:n]} & (p_n)_{n+i} & (p_n)_{n+j} & (p_n)_{[n+1:\infty)\setminus\{i,j\}} \\
(q_n)_{[1:n]} & (q_n)_{n+j} & (q_n)_{n+i} & (q_n)_{[n+1:\infty)\setminus\{i,j\}}
\end{array}\right].
\end{equation}

Next, we verify that $\|p_{n+1}-q_{n+1}\|_{\ell^2}\leq \|p_{n}-q_{n}\|_{\ell^2}$ by cases.
For~\eqref{eq.iter case 1}, equality holds.
For~\eqref{eq.iter case 2}, equality holds if ${(q_n)_{n+i}}=0$.
Otherwise, the condition $\sup a>0$ and $\sup b=0$ implies that $(p_n)_{n+i}$ and ${(q_n)_{n+i}}$ have opposite sign.
Thus,
\[
((p_n)_{n+i}-(q_n)_{n+i})^2
=(p_n)_{n+i}^2-2(p_n)_{n+i}(q_n)_{n+i}+((q_n)_{n+i})^2
\geq (p_n)_{n+i}^2+((q_n)_{n+i})^2,
\]
from which it follows that $\|p_{n+1}-q_{n+1}\|_{\ell^2}\leq \|p_{n}-q_{n}\|_{\ell^2}$.
The analysis for \eqref{eq.iter case 3} is identical.
For~\eqref{eq.iter case 4}, equality holds.
Finally, for~\eqref{eq.iter case 5}, we treat the case where $n$ is {odd}, as the {even} case is similar.
{Since $i\in\arg\max(a)$ and $j\in\arg\max(b)$, we have} $(p_n)_{n+i} \geq (p_n)_{n+j}$ and $(q_n)_{n+j} \geq (q_n)_{n+i}$.
Then $(p_n)_{n+i}((q_n)_{n+j} - (q_n)_{n+i}) \geq (p_n)_{n+j}((q_n)_{n+j} - (q_n)_{n+i})$, and rearranging gives
\[
(p_n)_{n+i}(q_n)_{n+j}+(p_n)_{n+j}(q_n)_{n+i}  \geq (p_n)_{n+i}(q_n)_{n+i}+(p_n)_{n+j}(q_n)_{n+j}.
\]
It follows that $\langle p_{n+1},q_{n+1}\rangle\geq\langle p_{n},q_{n}\rangle$, and since $\|p_{n+1}\|_{\ell^2}=\|p_n\|_{\ell^2}$ and $\|q_{n+1}\|_{\ell^2}=\|q_n\|_{\ell^2}$, we have $\|p_{n+1}-q_{n+1}\|_{\ell^2}\leq \|p_{n}-q_{n}\|_{\ell^2}$.

Finally, we verify that $p_n\to\Phi([p_0])$ and $q_n\to\Phi([q_0])$.
A case-by-case inductive argument establishes \eqref{eq.equal in initial tuple}, and so
\begin{align*}
\|p_n-\Phi([p_0])\|_{\ell^2}
&=\|(p_n)_{[n+1:\infty)}-\Phi([p_0])_{[n+1:\infty)}\|_{\ell^2}\\
&\leq\|(p_n)_{[n+1:\infty)}\|_{\ell^2}+\|\Phi([p_0])_{[n+1:\infty)}\|_{\ell^2}
\to0,
\end{align*}
and similarly, $\|q_n-\Phi([q_0])\|_{\ell^2}\to0$.
\end{proof}

While $c_2(\ell^2(\mathbb{N};\mathbb{R})\dslash S_\infty)=1$, we can use ideas from the proof of Theorem~2 in~\cite{AndoniNN:16} to show that $c_2(\ell^2(\mathbb{N};\mathbb{R}^d)\dslash S_\infty)=\infty$ for every $d\geq3$.
We do not know how to treat the $d=2$ case, but ideas in~\cite{AustinN:16} might be relevant.

\begin{theorem}
For each $d\geq 3$, it holds that $c_2(\ell^2(\mathbb{N};\mathbb{R}^d)\dslash S_\infty)=\infty$.
That is, $\ell^2(\mathbb{N};\mathbb{R}^d)\dslash S_\infty$ has no bilipschitz embedding into any Hilbert space.
\end{theorem}

\begin{proof}
Since $\ell^2(\mathbb{N};\mathbb{R}^3)\dslash S_\infty$ embeds isometrically into $\ell^2(\mathbb{N};\mathbb{R}^d)\dslash S_\infty$ for every $d>3$, it suffices to show $c_2(\ell^2(\mathbb{N};\mathbb{R}^3)\dslash S_\infty)=\infty$.
Fix a Hilbert space $H$ and $f\colon \ell^2(\mathbb{N};\mathbb{R}^3)\dslash S_\infty\to H$.
We will bound $\operatorname{dist}(f)$ with the help of a series of maps:
\[
X
\xrightarrow{\quad k \quad} Y
\xrightarrow{\quad h \quad} Z
\xrightarrow{\quad g \quad} \ell^2(\mathbb{N};\mathbb{R}^3)\dslash S_\infty
\xrightarrow{\quad f \quad} H.
\]
Given a constant-degree expander graph on $n$ vertices, we take $X$ to be the vertex set with path-distance metric.
Let $Y$ denote the same vertex set, but with the ``snowflaked'' metric defined by $d_Y(u,v)=d_X(u,v)^{1/2}$.
The map $k\colon X\to Y$ sends each vertex to itself.
Note that the optimal bilipschitz bounds of $k$ are $1$ and $\sqrt{\operatorname{diam}(X)}$.
Since $\operatorname{diam}(X)=O(\log n)$ by~\cite{Chung:89}, it follows that $\operatorname{dist}(k)=O(\sqrt{\log n})$.
Next, $Z$ is an $n$-point subspace of $\mathbb{R}^{3\times N}/S_N$ with quotient Frobenius metric, for some appropriately large $N$.
In particular, $h$ maps each $x_i\in Y$ to $S_N\cdot A_i$, where the columns of $A_i\in\mathbb{R}^{3\times N}$ are $1/\sqrt{n}$ times the members of $\mathcal{C}_i$ defined in equation~(10) in~\cite{AndoniNN:16}.
By Lemma~7 in~\cite{AndoniNN:16}, it holds that $\operatorname{dist}(h)\leq 2$ provided $N$ is appropriately large.
Let $R$ denote the maximum column norm in $\{A_i\}_{i=1}^n$, and select $v\in\mathbb{R}^3$ of norm $4R$.
For each $i\in\{1,\ldots,n\}$, define $u_i\in\ell^2(\mathbb{N};\mathbb{R}^3)$ by
\[
(u_i)_j=\left\{\begin{array}{cl} (A_i)_j+v & \text{if } j\leq N\\
0 & \text{if } j>N,\end{array}\right.
\]
where $(A_i)_j$ denotes the $j$th column of $A_i$.
Then $g\colon S_N\cdot A_i \mapsto [u_i]$ is an isometry.
Overall,
\[
\operatorname{dist}(f\circ g\circ h\circ k)
\leq \operatorname{dist}(f)\cdot \operatorname{dist}(g)\cdot\operatorname{dist}(h)\cdot\operatorname{dist}(k)
\lesssim \operatorname{dist}(f)\cdot \sqrt{\log n}.
\]
Meanwhile, Proposition~4.2 in~\cite{LinialLR:95} gives that $\operatorname{dist}(f\circ g\circ h\circ k)\gtrsim \log n$.
Combining with the above then gives $\operatorname{dist}(f)\gtrsim \sqrt{\log n}$.
By~\cite{LubotzkyPS:88}, we can take $n$ to be arbitrarily large, and so $\operatorname{dist}(f)=\infty$, as desired.
\end{proof}

\section{Quotients by translation}
\label{sec.translation}

Translation invariance is prevalent in signal and image processing.
For example, $\mathbb{Z}$ acts on $\ell^2(\mathbb{Z})$ by translation, and so one is inclined to consider the metric quotient $\ell^2(\mathbb{Z})\dslash\mathbb{Z}$.
In fact, $\ell^2(\mathbb{Z})\dslash\mathbb{Z}=\ell^2(\mathbb{Z})/\mathbb{Z}$ by taking $G:=X:=\mathbb{Z}$ and $\Omega:=\{0\}$ in the following.

\begin{lemma}
\label{lem: closed orbits}
Let $G$ be a second countable locally compact group acting freely on a measure space $X$ by measure space automorphisms.
Assume this action admits a measurable fundamental domain $\Omega \subseteq X$ with a measure $\nu$ such that $L^2(\Omega,\nu)$ is separable and the mapping $G \times (\Omega,\nu) \to X$ given by $(g,\omega) \mapsto g \cdot \omega$ is a measure space isomorphism.
Then $G$ has closed orbits in $L^2(X)$.
{In particular, $L^2(X)/G=L^2(X)\dslash G$ is a metric space.}
\end{lemma}

We postpone the proof of Lemma~\ref{lem: closed orbits}.
What follows is our main result in this section:

\begin{theorem}
\label{thm: l2ZmodZ}
For each $n \in \mathbb{N}$, consider a measure space $X_n$ along with a second countable locally compact abelian group $G_n$ that acts freely on $X_n$ by measure space isomorphisms and contains a discrete cyclic subgroup of order at least $n$.
Suppose this action admits a measurable fundamental domain $\Omega_n \subseteq X_n$ with a nonzero $\sigma$-finite measure $\nu_n$ such that $L^2(\Omega_n,\nu_n)$ is separable and the mapping ${G_n \times (\Omega_n,\nu_n)} \to X_n$ given by $(g,\omega)\mapsto g\cdot \omega$ is a measure space isomorphism.
Then
\begin{equation}
\label{eq: l2ZmodZ 1}
c_2\bigl( \ell^2(\mathbb{Z})/\mathbb{Z} \bigr) \leq \liminf_{n\to \infty} c_2\bigl( L^2(X_n)/G_n \bigr).
\end{equation}
\end{theorem}

Before proving Theorem~\ref{thm: l2ZmodZ}, we consider several examples.

\begin{example}
For each $n \in \mathbb{N}$, let $G_n := X_n := C_n \leq \mathbb{T}$ act on itself by multiplication, and take $\Omega_n := \{1\}$ for a fundamental domain.
Then Theorem~\ref{thm: l2ZmodZ} implies
\[
c_2\bigl( \ell^2(\mathbb{Z})/\mathbb{Z} \bigr) \leq \liminf_{n\to \infty} c_2\bigl( \ell^2(C_n)/C_n \bigr).
\]
\end{example}

\begin{example}
The rotation group $G := SO(2) \cong \mathbb{T}$ has discrete cyclic subgroups of all finite orders, and it acts freely on the punctured plane $X := \mathbb{R}^2\setminus \{0\}$.
Then the fundamental domain $\Omega := \{ (r,0) : r > 0 \}$ has measure $d\nu_n\bigl( (r,0) \bigr) := r\, dr$, and thanks to polar coordinates, there is a measure space isomorphism $G \times \Omega \cong X$.
(Here we scale Haar measure on $SO(2)$ to have total measure~$2\pi$.)
{Taking $G_n=G$ and $\Omega_n=\Omega$ for every $n$ in Theorem~\ref{thm: l2ZmodZ} gives}
\[
c_2\bigl( \ell^2(\mathbb{Z})/\mathbb{Z} \bigr) \leq c_2\bigl( L^2(\mathbb{R}^2)/SO(2) \bigr).
\]
\end{example}

\begin{example}
Let $\mathcal{G}$ be a second countable locally compact group, and let $G \leq \mathcal{G}$ be a closed abelian subgroup with arbitrarily large discrete cyclic subgroups.
(For instance, $\mathcal{G}=G=\mathbb{T}$, or $\mathcal{G} = \mathbb{R}$ and $G=\mathbb{Z}$.)
Consider the action of $G$ on $X:=\mathcal{G}$ by left multiplication.
By Theorem~3.6 in~\cite{Iverson:15}, there is a measurable fundamental domain $\Omega \subseteq \mathcal{G}$ with a $\sigma$-finite measure $\nu$ such that $L^2(\Omega,\nu)$ is separable and the mapping $G\times (\Omega,\nu) \to \mathcal{G}$ given by $(g,\omega)\mapsto g\omega$ is a measure space isomorphism.
{Then applying Theorem~\ref{thm: l2ZmodZ} with $G_n=G$ and $\Omega_n=\Omega$ for every $n$ gives} $c_2\bigl( \ell^2(\mathbb{Z})/\mathbb{Z} \bigr) \leq c_2\bigl( L^2(\mathcal{G})/G \bigr)$.
In particular, all of the following are bounded below by $c_2\bigl( \ell^2(\mathbb{Z})/\mathbb{Z} \bigr)$:
\[
c_2\bigl( L^2(\mathbb{T}^d)/\mathbb{T}^d \bigr),
\qquad
c_2 \bigl( L^2(\mathbb{R}^d)/\mathbb{R}^d \bigr),
\qquad
c_2 \bigl( L^2(\mathbb{R}^d)/\mathbb{Z}^d \bigr).
\]
\end{example}

These examples suggest a fundamental problem.

\begin{problem}
\label{prob.embed l2(Z)modZ}
Determine whether there exists a bilipschitz embedding of $\ell^2(\mathbb{Z})/\mathbb{Z}$ into some Hilbert space, and if so, find the exact value of $c_2\bigl( \ell^2(\mathbb{Z})/\mathbb{Z} \bigr)$.
\end{problem}

One may use the semidefinite program in Proposition~\ref{prop.sdp} to establish
\[
c_2\bigl( \ell^2(\mathbb{Z})/\mathbb{Z} \bigr)
\geq c_2(X)
\geq 1.3,
\]
where $X$ is a finite subspace of $\ell^2(\mathbb{Z})/\mathbb{Z}$.
This is the extent of our progress on Problem~\ref{prob.embed l2(Z)modZ}.

We first prove Lemma~\ref{lem: closed orbits}.
Our argument is inspired by~\cite{BarbieriHP:15}.
It involves spaces of vector-valued functions.
Given a measure space $(X,\mu)$ and a Hilbert space $H$, a function $\varphi \colon X \to H$ is \textbf{measurable} when the pre-image of every open set in $H$ is measurable in $X$.
Equivalently, for every $v \in H$, the complex-valued function $x \mapsto \langle \varphi(x), v \rangle$ is measurable on $X$~\cite{Pettis:38}.
We consider measurable functions $\varphi,\psi \colon X \to H$ to be equivalent if they differ only on a set of measure zero.
Then $L^2(X;H)$ is defined as the space of all equivalence classes of measurable functions $\varphi \colon X \to H$ for which $\int_X \| \varphi(x) \|_{H}^2 \, d\mu(x) < \infty$.
It is a Hilbert space with inner product $\langle \varphi, \psi \rangle := \int_X \langle \varphi(x), \psi(x)\rangle_{H} \, d\mu(x)$.
When $X$ is a group, one may define left translation in this space, and under the appropriate hypotheses, this operator is continuous:

\begin{proposition}
\label{prop: C0}
Let $G$ be a (multiplicative) second countable locally compact group, and let $H$ be a separable Hilbert space.
Given $g\in G$, let $L_g$ denote the left translation operator that transforms $\varphi\colon G \to X$ to $L_g \varphi \colon G \to X$ defined by $(L_g \varphi)(h) = \varphi(g^{-1} h)$ for $h \in G$.
The following hold for any $\varphi, \psi \in L^2(G;H)$.
	\begin{itemize}
	\item[(a)]
	The function $G \to L^2(G;H)$ that maps $g \mapsto L_g \varphi$ is continuous.
	\item[(b)]
	The function $G \to \mathbb{C}$ that maps $g \mapsto \langle L_g \varphi, \psi \rangle$ resides in $C_0(G)$.
	\end{itemize}
\end{proposition}

Proposition~\ref{prop: C0} is standard, but for the sake completeness, we give its proof in the appendix.
In~(a), it is vital that $H$ is separable.
For example, consider $G = \mathbb{R}$ and $H=\ell^2(\mathbb{R})$.
Choose any nonzero continuous function $f \in L^2(\mathbb{R})$, and let $\varphi \colon G \to H$ be the function with $\varphi(x) = f(x) \delta_x$, where $\delta_x \in \ell^2(\mathbb{R})$ is the point mass at $x \in \mathbb{R}$.
For any $y \neq 0$, it holds that $\langle L_y \varphi, \varphi \rangle = 0$, and thus $\| L_y \varphi - \varphi \| = 2 \| \varphi \| = 2 \| f \|$.
Hence, the function $y \mapsto L_y \varphi$ is discontinuous at~$0$.

\begin{proof}[Proof of Lemma~\ref{lem: closed orbits}]
We use multiplicative notation for the group operation in~$G$.
Given $g \in G$ and $\varphi \in L^2\big( G; L^2(\Omega,\nu) \big)$, let $L_g \varphi$ denote the left translation of $\varphi$ defined by $(L_g \varphi)(h)=\varphi(g^{-1}h)$ for $h \in G$.
We claim there is a unitary isomorphism $L^2(X) \to L^2\big( G; L^2(\Omega,\nu) \big)$ that converts the action of $G$ on $L^2(X)$ into left translation.
To see this, we first note that $L^2(G)$ is separable since $G$ is second countable; see (16.2) and (16.12) in~\cite{HewittR:63}.
Then Theorem~II.10 in~\cite{ReedS:72} gives an isomorphism $L^2(G \times \Omega) \to L^2\big( G; L^2(\Omega,\nu) \big)$ which, when precomposed with the isomorphism $L^2(X) \to L^2(G \times \Omega)$, produces a unitary $U \colon L^2(X) \to L^2\big( G; L^2(\Omega,\nu) \big)$ such that $(U f)(g)(\omega)=f(g\cdot \omega)$ for all $f \in L^2(X)$, $g \in G$, and $\omega \in \Omega$.
This proves the claim.

Thus, it suffices to show that $G$ has closed orbits in $L^2\big( G; L^2(\Omega,\nu) \big)$.
Suppose $\varphi,\psi \in L^2\big( G; L^2(\Omega,\nu) \big)$ and $\psi$ is in the closure of $\{ L_g \varphi : g \in G\}$.
We claim that $\psi = L_h \varphi$ for some $h \in G$.
This is clear when $\varphi = 0$, so we may assume otherwise.
For $g \in G$,
\begin{equation}
\label{eq: closed orbits 1}
\| L_g \varphi - \psi \|^2 
= \| \varphi \|^2 + \| \psi \|^2 - 2 \operatorname{Re} \langle L_g \varphi, \psi \rangle.
\end{equation}
Since the left-hand side of \eqref{eq: closed orbits 1} can be made arbitrarily small and $\|\varphi\|^2>0$ by assumption, there exists $g_0 \in G$ for which $\epsilon:= \operatorname{Re} \langle L_{g_0} \varphi, \psi \rangle > 0$.
By Proposition~\ref{prop: C0}, the function $g \mapsto \operatorname{Re} \langle L_g \varphi, \psi \rangle$ resides in $C_0(G)$.
In particular, there is a compact set $K \subseteq G$ such that $| \operatorname{Re} \langle L_g \varphi, \psi \rangle | < \epsilon/2$ for all $g \notin K$.
Then $g_0 \in K$, and
\[
\sup_{g\in G} \operatorname{Re} \langle L_g \varphi, \psi \rangle = \sup_{g\in K} \operatorname{Re} \langle L_g \varphi, \psi \rangle=:M.
\]
By the extreme value theorem, $M=\operatorname{Re} \langle L_h \varphi, \psi \rangle$ for some $h \in K$.
Then~\eqref{eq: closed orbits 1} implies
\[
\| L_h \varphi - \psi \|^2 = \inf_{g\in G} \| L_g \varphi - \psi \|^2 = 0.
\qedhere
\]
\end{proof}

\begin{proof}[Proof of Theorem~\ref{thm: l2ZmodZ}]
Consider the dense subset
\[
c_{00}(\mathbb{Z})/\mathbb{Z}
= \Big\{ [a] \in \ell^2(\mathbb{Z})/\mathbb{Z} : a_k = 0 \text{ for all but finitely many }k\in \mathbb{Z} \Big\}.
\]
Given finite $F \subseteq c_{00}(\mathbb{Z})/\mathbb{Z}$, select $N \in \mathbb{N}$ such that each $[a] \in F$ has a representative with support contained in $\{0,\dotsc,N\}$.
For each $n > 2N$, we will construct an isometric embedding $\phi \colon F \to L^2(X_n)/G_n$, implying $c_2(F) \leq c_2\big( L^2(X_n)/G_n \big)$.
Then Lemma~\ref{lem.euclidean distortion dense} gives
\[
c_2\big( \ell^2(\mathbb{Z})  / \mathbb{Z} \big)
= \sup\{c_2(F) : F \subseteq c_{00}(\mathbb{Z})/\mathbb{Z}, |F|<\infty \}
\leq \liminf_{n\to \infty}  c_2\big( L^2(X_n)/G_n \big).
\]

First, we construct a critical set $B \subseteq X_n$, whose normalized indicator function will emulate a point mass $\delta_0 \in \ell^2(\mathbb{Z})$.
Let $\mu$ be the measure in $X_n$.
Express the group operation in $G_n$ multiplicatively, and write $|S|$ for the Haar measure of Borel $S \subseteq G_n$.
By hypothesis, there exists $h \in G_n$ that generates a discrete subgroup $H \leq G_n$ of order at least $2N+1$.
Then there is a neighborhood $U\subseteq G_n$ of $1$ that fails to intersect $H \setminus \{1\}$.
By Proposition~2.1 of~\cite{Folland:95}, $U$ contains a neighborhood $V \subseteq G_n$ of $1$ such that $VV \subseteq U$, and $V$ in turn contains a compact neighborhood $W \subseteq G_n$ of $1$ such that $W^{-1} = W$ and $WW \subseteq V$.
{(While Proposition~2.1 of~\cite{Folland:95} does not directly provide compactness of $W$, this may be obtained by replacing $W$ with a compact neighborhood of $1$ contained in $W$ and then intersecting with its inverse.)}
Since $W$ is {compact and contains an open set}, it has finite positive measure {by Proposition~2.19 of~\cite{Folland:95}}.
Meanwhile, since $\nu_n$ is $\sigma$-finite and nonzero, there exists measurable $E \subseteq \Omega_n$ with $0 < \nu_n(E) < \infty$.
Put 
\[
B
:=\{ g\cdot x : g \in W, x \in E \}
\subseteq X_n.
\]
Then $\mu(B) = |W| \nu(E) \in (0,\infty)$ by the isomorphism $G_n \times (\Omega_n,\nu_n) \cong X_n$.

Next, we claim that
\begin{equation}
\label{eq: disjointness of translates}
\text{for every $g \in G_n$, there is at most one $k \in H$ for which $g \cdot B$ intersects $k \cdot B$.}
\end{equation}
Indeed, suppose $g \cdot B$ intersects $k_1 \cdot B$ and $k_2 \cdot B$ for $k_1,k_2 \in H$.
Then for each $i \in \{1,2\}$, there exist $v_i,w_i \in W$ and $x_i,y_i \in \Omega_n$ with $g v_i \cdot x_i = k_i w_i \cdot y_i$.
Since $\Omega_n$ is a fundamental domain for the free action of $G_n$, it follows that $x_i = y_i$ and $g v_i = k_i w_i$.
Then $k_1 w_1 v_1^{-1} = g = k_2 w_2 v_2^{-1}$, and $k_2^{-1} k_1 = w_2 v_2^{-1} v_1 w_1^{-1}$.
By the conditions on $U$, $V$, and $W$, we conclude that $k_2^{-1} k_1 \in U$ and $k_1 = k_2$.

With the aid of $B$, we now define an isometry $\phi \colon F \to L^2(X_n)/G_n$.
Denote the indicator function of a set $S$ by $\mathbf{1}_S$.
For each $[a] \in F$, fix a representative $a^* \in c_{00}(\mathbb{Z})$ with support contained in $\{0,\dotsc,N\}$.
Then define 
\[ \varphi(a^*) := \mu(B)^{-1/2} \sum_{i=0}^N a_i^* \mathbf{1}_{h^i \cdot B} \in L^2(X_n) \]
and $\phi([a]) := [\varphi(a^*)] \in L^2(X_n)/G_n$.

In the remainder of the proof, we verify $\phi$ preserves max filters, and is therefore an isometry.
For $[a],[b] \in F$, we have
\begin{align}
\label{eq: max filter as sum}
\llangle \phi([a]), \phi([b]) \rrangle
\notag &= \sup_{g \in G_n} \operatorname{Re} \bigg\langle \mu(B)^{-1/2} \sum_{i=0}^N a_i^* \mathbf{1}_{g h^i \cdot B}, \, \mu(B)^{-1/2} \sum_{j=0}^N b_j^* \mathbf{1}_{h^j \cdot B} \bigg\rangle \\
\notag &= \sup_{g\in G_n} \operatorname{Re} \sum_{i,j=0}^N \overline{a_i^*} b_j^* \frac{\mu(gh^i \cdot B \cap h^j \cdot B)}{\mu(B)} \\
&= \sup_{g\in G_n} \operatorname{Re} \sum_{i,j \in \mathbb{Z}} \overline{a_i^*} b_j^* \frac{\mu(B \cap gh^{i-j} \cdot B)}{\mu(B)},
\end{align}
where the last step uses the facts that $G_n$ is abelian and preserves measure in $X_n$.

For a lower bound, \eqref{eq: max filter as sum} implies
\[
\llangle \phi([a]), \phi([b]) \rrangle
\geq \sup_{k \in \mathbb{Z}} \operatorname{Re} \sum_{i\in \mathbb{Z}} \sum_{j\in \mathbb{Z}} \overline{a_i^*} b_j^* \frac{\mu(B \cap h^{i-j+k} \cdot B)}{\mu(B)}.
\]
In the inner sum, every term vanishes besides the one with $j=i+k$, by~\eqref{eq: disjointness of translates}.
Therefore,
\[
\llangle \phi([a]), \phi([b]) \rrangle 
\geq \sup_{k \in \mathbb{Z}} \operatorname{Re} \sum_{i\in \mathbb{Z}} \overline{a_i^*} b_{i+k}^*
= \llangle [a], [b] \rrangle.
\]

To prove the matching upper bound, we first write~\eqref{eq: max filter as sum} as
\begin{equation}
\label{eq: max filter as sum 2}
\llangle \phi([a]), \phi([b]) \rrangle
= \sup_{g\in G_n} \operatorname{Re} \sum_{j=0}^N \sum_{i=0}^N \overline{a_i^*} b_j^* \frac{\mu(g\cdot B \cap h^{j-i} \cdot B)}{\mu(B)}.
\end{equation}
For each $g \in G_n$, apply~\eqref{eq: disjointness of translates} to choose an index $c(g) \in \{-N,\dotsc,N\}$ such that $g\cdot B \cap h^i B = \varnothing$ for every $i \in \{-N,\dotsc,N\}$ with $i \neq c(g)$.
(Here we use the fact that $h$ has order at least $2N+1$.)
Then every term in the inner sum of~\eqref{eq: max filter as sum 2} vanishes besides the one with $i=j-c(g)$ when it is present, so that
\begin{align*}
\llangle \phi([a]), \phi([b]) \rrangle
&= \sup_{g\in G_n} \operatorname{Re} \sum_{j=0}^N \overline{a_{j-c(g)}^*} b_j^* \frac{\mu(g\cdot B \cap h^{c(g)} \cdot B)}{\mu(B)} \\
&= \sup_{g\in G_n} \frac{\mu(B \cap g^{-1}h^{c(g)} \cdot B)}{\mu(B)} \operatorname{Re} \sum_{j\in \mathbb{Z}} \overline{a_{j-c(g)}^*} b_j^* .
\end{align*}
Denoting $t_+$ for the positive part of a number $t \in \mathbb{R}$, we have
\[
\llangle \phi([a]), \phi([b]) \rrangle
\leq \sup_{g \in G_n} \bigg( \operatorname{Re} \sum_{j\in \mathbb{Z}} \overline{a_{j-c(g)}^*} b_j^* \bigg)_+
\leq \sup_{k \in \mathbb{Z}} \bigg( \operatorname{Re} \sum_{j\in \mathbb{Z}} \overline{a_{j-k}^*} b_j^* \bigg)_+ .
\]
Since $a^*,b^* \in c_{00}(\mathbb{Z})$, there exists $k \in \mathbb{Z}$ for which $\sum_{j\in \mathbb{Z}} \overline{ a_{j-k}^* } b_j^* = 0$.
Consequently,
\[
\llangle \phi([a]), \phi([b]) \rrangle
\leq \sup_{k \in \mathbb{Z}} \operatorname{Re} \sum_{j\in \mathbb{Z}} \overline{a_{j-k}^*} b_j^*
= \llangle [a], [b] \rrangle.
\]
This completes the proof.
\end{proof}

\section*{Acknowledgments}

The authors thank Amit Singer and Ramon van Handel for a lively discussion involving a Drake meme that prompted this investigation, {as well as the anonymous reviewers for several pages of helpful and detailed suggestions that substantially improved the presentation of our results.}

\appendix

\section{Proof of Proposition~\ref{prop.lagrange inperpolation}}

First, we consider the special case in which the coordinates $(u_1)_k,\ldots,(u_n)_k$ are distinct for each $k\in\{1,\ldots,d\}$.
For this case, we construct $p\colon\mathbb{R}^d\to\mathbb{R}$ of the form
\begin{equation}
\label{eq.poly form}
p(x_1,\ldots,x_d)
=\sum_{k=1}^d p_k(x_k),
\end{equation}
where each $p_k$ is a polynomial in one variable.
Then for each $k\in\{1,\ldots,d\}$, we seek a polynomial $q_k:=p_k'$ that satisfies
\[
q_k((u_i)_k)
=\frac{\partial p}{\partial x_k}(u_i)
=(v_i)_k
\]
for all $i\in\{1,\ldots,n\}$.
This can be accomplished by Lagrange interpolation, and we integrate to get $p_k$, and then \eqref{eq.poly form} gives the desired polynomial.

Next, we generalize the above argument to treat the general case.
To do so, we select an invertible matrix $A\in\mathbb{R}^{d\times d}$ in which a way that the coordinates $(Au_1)_k,\ldots,(Au_n)_k$ are distinct for each $k\in\{1,\ldots,d\}$.
Equivalently, $A\in\mathbb{R}^{d\times d}$ satisfies $\operatorname{det}(A)\neq0$ and $e_k^\top A(u_i-u_j)\neq0$ for every $k\in\{1,\ldots,d\}$ and every $i,j\in\{1,\ldots,n\}$ with $i\neq j$.
That is, it suffices for $A$ to be generic in the sense that it does not reside in the zero set of a nonzero polynomial.
(In particular, such an $A$ necessarily exists.)
Consider $\overline{u}_i:=Au_i$ and $\overline{v}_i:=(A^\top)^{-1}v_i$ for all $i\in\{1,\ldots,n\}$.
Since the coordinates {$(\overline{u}_1)_k,\ldots,(\overline{u}_n)_k$} are distinct for each $k\in\{1,\ldots,d\}$, our treatment of the previous case delivers a polynomial function $\overline{p}$ such that $\nabla\overline{p}(\overline{u}_i)=\overline{v}_i$ for all $i\in\{1,\ldots,n\}$.
Denoting $L\colon x\mapsto Ax$ and $p:=\overline{p}\circ L$, then since the Jacobian of $\overline{p}$ at $x$ is $\nabla \overline{p}(x)^\top$, the chain rule gives
\[
v_i^\top
=\overline{v}_i^\top A
=\nabla \overline{p}(\overline{u}_i)^\top A
=\nabla(\overline{p}\circ L)(u_i)^\top
=\nabla p(u_i)^\top,
\]
that is, $\nabla p(u_i)=v_i$ for every $i\in\{1,\ldots,n\}$, as desired.

\section{Proof of Proposition~\ref{prop.quotient immersion}}

{
First, the finite group $G$ acts properly on $M$ by Proposition~21.5 in~\cite{Lee:13}.
Parts (a) and (b) are given by Theorem~21.10 in~\cite{Lee:13}.
For~(c), Theorem~4.29 in~\cite{Lee:13} gives that $g^\downarrow$ is smooth.
To see it is an immersion, we first apply} the chain rule to $g=g^\downarrow\circ \pi$ at an arbitrary point $x\in M$:
\begin{equation}
\label{eq.chain rule for immersion}
Dg(x) = Dg^\downarrow([x])\circ D\pi(x).
\end{equation}
Since $g$ is an immersion, $Dg(x)$ is injective.
Since $\pi$ is a submersion, $D\pi(x)$ is surjective.
{By part~(a), the domain and codomain of $D\pi(x)$ have the same dimension.
Consequently, $D\pi(x)$ is a bijection.
By~\eqref{eq.chain rule for immersion},
\[
Dg^\downarrow([x])
=Dg(x)\circ(D\pi(x))^{-1}
\]
is injective, and so $g^\downarrow$ is an immersion.
}

{
Next, (d) is given by Proposition~2.32 in~\cite{Lee:18}.
For~(e), let $d_R$ denote the Riemannian distance in $M/G$, that is
\[
d_R([x],[y])
=\inf_\gamma \int_I \|\gamma'(t)\|dt,
\]
where the infimum is over all admissible curves $\gamma\colon I\to M/G$ from $[x]$ to $[y]$.
Next, let $d_Q=d_{M/G}$ denote the quotient metric given by
\[
d_Q([x],[y])
=\inf_{\substack{p\in[x]\\q\in[y]}}d_M(p,q),
\]
where $d_M$ is the Riemannian distance in $M$.
Then
\begin{equation}
\label{eq.quotient riemannian}
d_Q([x],[y])
=\inf_\gamma\int_I\|\gamma'(t)\|dt,
\end{equation}
where the infimum is over all admissible curves in $M$ beginning in the set $[x]$ and ending in the set $[y]$.
Our task is to show that $d_R([x],[y])=d_Q([x],[y])$.
}

{
First, we show that $d_R([x],[y])\leq d_Q([x],[y])$.
To this end, it suffices to show that $\pi\colon (M,d_M)\to(M/G,d_R)$ is $1$-Lipschitz, since this implies
\[
d_R([x],[y])
=d_R(\pi(x),\pi(y))
=d_R(\pi(p),\pi(q))
\leq d_M(p,q)
\]
for every $p\in[x]$ and $q\in[y]$, and the desired bound follows from taking an infimum.
To show that $\pi$ is $1$-Lipschitz, take any admissible curve $\gamma\colon I\to M$ from $x$ to $y$.
Then by Proposition~2.47(c) in~\cite{Lee:18}, $\pi\circ\gamma \colon I\to M/G$ is an admissible curve from $[x]$ to $[y]$ of the same length.
In particular, the length of any such $\gamma$ is an upper bound on $d_R([x],[y])$, and taking an infimum proves the claim.
}

{
Finally, we show that $d_R([x],[y])\geq d_Q([x],[y])$.
Choose $\epsilon>0$ and select an admissible curve $\gamma\colon I\to M/G$ from $[x]$ to $[y]$ of length at most $d_R([x],[y])+\epsilon$.
Then there exists an admissible curve $\gamma^\uparrow\colon I\to M$ such that $\gamma=\pi\circ\gamma^\uparrow$.
(See, for instance, Exercise~3.12.6 in~\cite{Hamilton:17}.)
As above, $\gamma^\uparrow$ has the same length as $\gamma$.
By~\eqref{eq.quotient riemannian}, it follows that
\[
d_Q([x],[y])
\leq d_R([x],[y])+\epsilon.
\]
Since $\epsilon>0$ is arbitrary, the desired bound follows.
}

\section{Proof of Proposition~\ref{prop.embedding is bilipschitz}}

We present an argument from Moishe Kohan~\cite{Kohan:math.se}.
{First, we show that any smooth map $f\colon X\to Y$ is upper Lipschitz.}
Take any $p,q\in X$ and let $c\colon[0,d_X(p,q)]\to X$ denote the unit-speed parameterization of a shortest path from $p$ to $q$.
Then
\begin{align*}
d_Y(f(p),f(q))
&{\leq}\int_0^{d_X(p,q)}\|(f\circ c)'(t)\|dt
=\int_0^{d_X(p,q)}\|(Df)(c(t))c'(t)\|dt\\
&\leq \sup_{x\in X} \|(Df)(x)\|\cdot \int_0^{d_X(p,q)}\|c'(t)\|dt
=\sup_{x\in X} \|(Df)(x)\|\cdot d_X(p,q).
\end{align*}
Furthermore, $\beta:=\sup_{x\in X} \|(Df)(x)\|<\infty$ by compactness.

{To prove lower Lipschitz, we further assume that $f\colon X\to Y$ is an embedding.
To this end, we first} show that $f$ is \textit{locally lower Lipschitz} in the sense that there exist $\alpha,\epsilon>0$ such that $d_Y(f(p),f(q))\geq \alpha d_X(p,q)$ whenever $d_X(p,q)<\epsilon$.
{To accomplish this, we take $Z:=\operatorname{im}(f)$ and show that the inverse map $f^{-1}\colon (Z,d_Y)\to (X,d_X)$ is $\delta$-locally $L$-Lipschitz, in which case we may take $\alpha:=1/L$ and $\epsilon:=\delta/\beta$.
Our approach uses Riemannian tools to analyze arbitrary $u,v\in Z$ with $d_Y(u,v)$ sufficiently small, and so we start by constructing a compact submanifold with boundary $K\subseteq Y$ that contains $Z$ along with all length-minimizing geodesics in $Y$ between sufficiently close points in $Z$.
Since $Z$} is a compact submanifold of $Y$, it has a \textit{normal injectivity radius} $r>0$.
That is, letting $B_r\nu_Z$ denote the members of the normal bundle $\nu_Z$ of $Z$ in $Y$ with norm smaller than $r$, then the restriction of the normal exponential map $\operatorname{exp}_Z\colon\nu_Z\to Y$ to $B_r\nu_Z$ is a diffeomorphism onto its image $N\subseteq Y$, which in turn is an open tubular neighborhood of $Z$.
Invert this diffeomorphism and project onto $Z$ to obtain a smooth retraction ${g}\colon N\to Z$.
Applying $\operatorname{exp}_Z$ to the closure of $B_{r/2}\nu_Z$ gives a compact submanifold with boundary $K\subseteq N$.
{Next, we find $\delta>0$ such that $d_K(u,v)=d_Y(u,v)$ for every $u,v\in Z$ with $d_K(u,v)<\delta$.
By Theorem~6.17 in~\cite{Lee:18}, every $z\in Z$ is contained in an open set $B_z\subseteq K$ that is \textit{geodesically convex}.
This means that for every $u,v\in B_z$, the distance $d_Y(u,v)$ equals the length of the unique shortest path in $Y$ from $u$ to $v$, which in turn is entirely contained in $B_z\subseteq K$, and so $d_K(u,v)=d_Y(u,v)$. 
Considering $\{B_z\}_{z\in Z}$ is an open cover of the compact metric space $(Z,d_K)$, then Lebesgue's number lemma produces $\delta>0$ such that every $u,v\in Z$ with $d_K(u,v)<\delta$ necessary satisfies $u,v\in B_z$ for some $z\in Z$, in which case $d_K(u,v)=d_Y(u,v)$.
The previous argument then gives that the restriction of $f^{-1}\circ g\colon N\to X$ to $Z$ is $\delta$-locally $L$-Lipschitz for some $L<\infty$, and so $f^{-1}$ is $\delta$-locally $L$-Lipschitz, as desired.}

Finally, consider the set $C$ of $(p,q)\in X^2$ for which $d_X(p,q)\geq\epsilon$.
Then compactness gives
\[
R:=\sup_{(p,q)\in C}\frac{d_X(p,q)}{d_Y(f(p),f(q))}
<\infty.
\]
Since $f$ is $\epsilon$-locally $\alpha$-lower Lipschitz, it follows that $f$ is $\min\{\alpha,1/R\}$-lower Lipschitz.

\section{Proof of Proposition~\ref{prop.generic projection of semialgebraic manifold}}

First, we note that $L|_M$ is not injective precisely when there exist $x,y\in M$ with $x\neq y$ such that $L(x-y)=0$.
Similarly, $L|_M$ is not a {smooth} immersion precisely when there exists $x\in M$ and a nonzero $v\in T_xM$ such that $Lv=0$.
Letting $W$ denote the set of all such $x-y$ and $v$, then {by Proposition~4.22 in~\cite{Lee:13},} $L|_M$ is not a {smooth} embedding precisely when there exists $w\in W$ such that $Lw=0$.
Consider the set $P$ of $(L,w)$ for which $Lw=0$ and $w\in W$.
Then $L|_M$ is a {smooth} embedding precisely when $L$ does not reside in the projection $\pi_L(P)$.
{Lemma~2.8} in~\cite{DymG:22} gives
\[
\operatorname{dim}\pi_L(P)
\leq \operatorname{dim}P
\leq \operatorname{dim}W+\sup_{w\in W}\operatorname{dim}\{L:(L,w)\in P\}.
\]
Meanwhile, {$\operatorname{dim}W\leq2k$} and $\operatorname{dim}\{L:Lw=0\}=m(n-1)$ for every $w\neq0$.
Thus, $\operatorname{dim}\pi_L(P)\leq 2k+m(n-1)$, which is strictly less than $mn$ when $m>2k$.
It follows that a generic choice of $L$ does not reside in $\pi_L(P)$.

\section{Proof of Proposition~\ref{prop: C0}}

For (a), choose $\epsilon > 0$.
Given $f \in L^2(G)$ and $u \in H$, let $f \otimes u \in L^2(G;H)$ denote the function defined by $(f \otimes u)(g) = f(g)u$ for $g \in G$.
Such pure tensors span a dense subspace of $L^2(G; H)$ by Theorem~II.10 of~\cite{ReedS:72}.
(Here, we use the separability of $L^2(G)$; see statements~(16.2) and~(16.12) in~\cite{HewittR:63}.)
Find $f_1,\dotsc,f_m \in L^2(G)$ and unit norm $u_1,\dotsc,u_m \in H$ with $\| \varphi - \sum_{j=1}^m f_j \otimes u_j \| {< \frac{\epsilon}{3}}$.
By Proposition~2.41 of~\cite{Folland:95}, there is a neighborhood $V$ of $1$ in $G$ such that ${\| L_g f_j - f_j \|} {< \frac{\epsilon}{3m} }$ whenever $g \in V$.
For such $g$, the triangle inequality implies
\begin{align*}
\| L_g \varphi - \varphi \| 
&\leq \bigg\| L_g \Big(\varphi - \sum_{j=1}^m f_j \otimes u_j \Big) \bigg\| 
+\bigg\| \sum_{j=1}^m (L_g f_j - f_j) \otimes u_j \bigg\|
+ \bigg\| \sum_{j=1}^m f_j \otimes u_j - \varphi \bigg\| \\
&\leq 2 \bigg\| \varphi - \sum_{j=1}^m f_j \otimes u_j \bigg\| + \sum_{j=1}^m \| L_g f_j - f_j \| < \epsilon.
\end{align*}
Then for $g,h \in G$ with $h \in gV$, it holds that $\| L_h \varphi - L_g \varphi \| = \| L_{g^{-1}h} \varphi - \varphi \| {< \epsilon}$.

Part~(b) is trivial when $\varphi = 0$ or $\psi=0$, and so we assume otherwise.
Continuity follows from~(a), since 
\[ 
|\langle L_h \varphi, \psi \rangle - \langle L_g \varphi, \psi \rangle |
\leq \| L_h \varphi - L_g \varphi \| \| \psi \|.
\]
It remains to prove our function vanishes at infinity.

Given a subset $E \subseteq G$, let $\mathbf{1}_E$ denote its indicator function.
We claim that for every $\epsilon > 0$, there is a compact set $K \subseteq G$ such that $\| \mathbf{1}_{K^c} \varphi \| < \epsilon$.
Indeed, the function $g \mapsto \| \varphi(g)\|$ belongs to $L^2(G)$, so we can approximate it within $\epsilon$ by some $f \in C_c(G)$.
Let $K \subseteq G$ be a compact set outside of which $f$ vanishes.
Then
\[
\| \mathbf{1}_{K^c} \varphi \|^2
\leq \int_{K^c} \| \varphi(g) \|^2\, dg + \int_K \big| \| \varphi(g)\| - f(g) \big|^2\, dg
= \int_G \big| \| \varphi(g) \| - f(g) \big|^2\, dg
< \epsilon^2,
\]
as desired.

Choose $\epsilon > 0$.
By the above, there are compact sets $K_1,K_2 \subseteq G$ such that $\| \mathbf{1}_{K_1^c} \varphi\| < \frac{\epsilon}{2\| \psi \|}$ and $\| \mathbf{1}_{K_2^c} \psi \| < \frac{\epsilon}{2 \| \varphi \|}$.
Then
\[ K_2 K_1^{-1} = \{ gh^{-1} : g \in K_2, h \in K_1 \} \subseteq G \]
is compact.
Given $g \not\in K_2 K_1^{-1}$, we show $|\langle L_g \varphi, \psi \rangle| < \epsilon$.
Expand
\begin{align*}
\langle L_g \varphi, \psi \rangle &= 
\big\langle L_g \varphi, \mathbf{1}_{K_2} \psi \big\rangle 
+ \big\langle L_g \varphi, \mathbf{1}_{K_2^c} \psi \rangle \\
&=\big\langle \mathbf{1}_{K_1} \varphi, \mathbf{1}_{g^{-1} K_2} L_{g^{-1}} \psi \big\rangle 
+ \big\langle \mathbf{1}_{K_1^c} \varphi, L_{g^{-1}}( \mathbf{1}_{K_2} \psi ) \big\rangle
+ \big\langle L_g \varphi, \mathbf{1}_{K_2^c} \psi \big\rangle.
\end{align*}
Since $g \not\in K_2K_1^{-1}$, the sets $K_1$ and $g^{-1} K_2$ are disjoint.
Then the first term above vanishes, and the triangle inequality yields
\begin{align*}
|\langle L_g \varphi, \psi \rangle| 
&\leq \big| \big\langle \mathbf{1}_{K_1^c} \varphi, L_{g^{-1}}( \mathbf{1}_{K_2} \psi ) \big\rangle \big| 
+ \big| \big\langle L_g \varphi, \mathbf{1}_{K_2^c} \psi \big\rangle \big| \\
&\leq \| \mathbf{1}_{K_1^c} \varphi \| \| L_{g^{-1}}( \mathbf{1}_{K_2} \psi ) \| + \| L_g \varphi \| \| \mathbf{1}_{K_2^c} \psi \| \\
&= \| \mathbf{1}_{K_1^c} \varphi \| \|  \mathbf{1}_{K_2} \psi \| + \| \varphi \| \| \mathbf{1}_{K_2^c} \psi \|
\leq \| \mathbf{1}_{K_1^c} \varphi \| \|  \psi \| + \| \varphi \| \| \mathbf{1}_{K_2^c} \psi \| < \epsilon,
\end{align*}
as desired.

\end{document}